\documentclass[reqno,11pt]{amsart}  

\usepackage[foot]{amsaddr}

\usepackage{amsmath}
\usepackage{amssymb} 
\usepackage{mathtools} 
\usepackage{graphicx}
\usepackage{graphics}  
\usepackage{xcolor}
\usepackage{verbatim} 
\usepackage[normalem]{ulem} 
\usepackage{upgreek}
\usepackage{enumerate} 
\usepackage{cite}
\usepackage{amsmath}
\usepackage{amsfonts}
\usepackage{amssymb}
\usepackage{bbm}
\usepackage{cancel}
\usepackage{graphicx}
\usepackage{multicol}
\usepackage[utf8]{inputenc}
\usepackage{framed}
\usepackage{setspace}
\usepackage{amsthm}
\usepackage{hyperref}
\hypersetup{
  colorlinks   = true, 
  urlcolor     = blue, 
  linkcolor    = blue, 
  citecolor   = blue 
}
\usepackage{caption}
\usepackage{subcaption}
\usepackage{bbm}
\usepackage{graphicx}
\usepackage[titletoc]{appendix}
\numberwithin{equation}{section}  
 
\usepackage[refpage,noprefix]{nomencl}
\usepackage{nomencl} 
\usepackage{enumitem}

\newtheorem{theorem}{Theorem}[section] 
\newtheorem{lemma}[theorem]{Lemma} 
\newtheorem{proposition}[theorem]{Proposition} 
\newtheorem{cor}[theorem]{Corollary}

\theoremstyle{definition}

\theoremstyle{remark}
\newtheorem{remark}[theorem]{Remark} 

\usepackage{float}
\restylefloat{figure}

\DeclareMathAlphabet{\mathpzc}{OT1}{pzc}{m}{it}

\newcommand{\abs}[1]{\left| #1 \right|}

\renewcommand{\L} {\Lambda} %

\def\d{\delta} 
\newcommand{\e} {\varepsilon} 
\newcommand{\eps}{\varepsilon} 
\renewcommand{\epsilon}{\varepsilon}

\def\l{\lambda}

\renewcommand{\theta}{\vartheta}

\newfam\Bbbfam 
\font\tenBbb=msbm10 
\font\sevenBbb=msbm7 
\font\fiveBbb=msbm5 
\textfont\Bbbfam=\tenBbb 
\scriptfont\Bbbfam=\sevenBbb 
\scriptscriptfont\Bbbfam=\fiveBbb

\newcommand{\R}     {\mathbb{R}} 
\newcommand{\Z}     {\mathbb{Z}} 
\newcommand{\N}     {\mathbb{N}} 
\renewcommand{\P}   {\mathbb{P}} 
 
\newcommand{\E}     {\mathbb{E}}

\newcommand{\smfrac}[2]{{\textstyle{\frac {#1}{#2}}}}

\def\1{{\mathchoice {1\mskip-4mu\mathrm l}      
{1\mskip-4mu\mathrm l} 
{1\mskip-4.5mu\mathrm l} {1\mskip-5mu\mathrm l}}} 
\newcommand{\ssup}[1] {{\scriptscriptstyle{({#1}})}} 
\newcommand{\Ssup}[1]{^{\ssup{#1}}}
\newcommand{\sqsup}[1]{{\scriptscriptstyle{[{#1}}]}} 
\def\comment#1{} 
\newtheoremstyle{thm}{2ex}{2ex}{\itshape\rmfamily}{} 
{\bfseries\rmfamily}{}{1.7ex}{} 
\newtheoremstyle{rem}{1.3ex}{1.3ex}{\rmfamily}{} 
{\itshape\rmfamily}{}{1.5ex}{} 
 
\newcommand{\bc} {\boldsymbol{c}}

\newcommand{\Ocal}   {{\mathcal O }}

\newcommand{\ex}{{\rm e}} 
 
\renewcommand{\d}{{\rm d}} 
\newcommand{\per}{{\text{\rm per}}} 
\newcommand{\diff}{{\text{\rm diffusive}}}

\newcommand{\Sym}{\mathfrak{S}}

\newcommand{\dist}{{\operatorname {dist}}}

\newcommand{\Exp}{\mathscr{E}\kern-0.2mm{\operatorname{xp}}}
\newcommand{\Log}{\mathscr{L}\kern-0.2mm{\operatorname{og}}}

\newcommand\NoBlackBoxes{\global\overfullrule0pt}
\NoBlackBoxes

\newcommand{\Sfrak}{{\mathfrak{S}}}

\newcommand\mycom[2]{\genfrac{}{}{0pt}{}{#1}{#2}}

\newcommand{\ek}[1]{\left[#1\right]}
\newcommand{\rk}[1]{\left(#1\right)}

\newcommand{\M}{{j}}
\newcommand{\partition}{m}
\newcommand{\bcd}{u} 
\newcommand{\free}{{\mathrm{free}}}
\newcommand{\dir}{{\mathrm{Dir}}}
\newcommand{\neu}{{\mathrm{Neu}}}
\renewcommand{\bc}{\mathrm{bc}}

\newcommand{\poi}[1]{\mathsf{Poi}_{#1}}

\newcommand{\Pfrak}{\mathfrak{P}}

\newcommand{\hk}[1]{^{(#1)}}

\newcommand{\PPP}{\mathsf{P}\!\mathsf{P}\!\mathsf{P}}

\newcommand{\gk}[1]{\left\{#1\right\}}

\newcommand{\al}{\abs{\L}}
\newcommand{\alN}{\abs{\L_N}}

\newcommand*{\rhoc}{\rho_{\rm c}}
\newcommand*{\rhoe}{\rho_{\rm e}}

\newcommand{\gammabc}[1]{\gamma_N^{\ssup{\L_N,{\mathrm{#1}}}}}
\renewcommand*{\PPP}{{\tt{P}}}
\usepackage{xifthen}

\newcommand{\optarg}[1][]{%
  \ifthenelse{\isempty{#1}}%
    {{\tt{P}}_\L}
    {{\tt{P}}_\L^{\ssup{{\rm #1}}}}
}

\newcommand*{\EPP}{{\tt{E}}}

\newcommand{\defeq}{\vcentcolon=}

\newcommand{\q}{\mathrm{t}}
\newcommand{\qbc}{\q^{\ssup{\mathsf{bc}}}}

\newcommand{\Pfree}{\mathtt{P}^{\ssup{\free}}}

\newcommand{\gfree}{g^{\ssup{\free}}}

\newcommand{\densfree}{\rho^{\ssup{\free}}}

\newcommand{\Pbc}{\mathtt{P}^{\ssup{\bc,N}}}
\newcommand{\Ebc}{\mathtt{E}^{\ssup{\bc}}}
\newcommand{\gbc}{g^{\ssup{\mathsf{bc}}}}

\newcommand*{\pres}{p}
\newcommand{\presbc}{\pres^{\ssup{\bc}}}

\newcommand{\presbcN}{\pres^{\ssup{\bc,N}}} 
\newcommand{\densbcN}{\rho^{\ssup{\bc,N}}} 

\newcommand{\Nrm}{\mathrm{N}}
\newcommand{\Nop}{\mathrm{N}^{\ssup{\1}}}
\newcommand{\presfree}{\pres^{\ssup{\free}}}

\newcommand{\Partn}{\mathrm{N}}
\newcommand{\Ns}{\Nrm^{\ssup{\mathrm{short}}}}
\newcommand{\Nl}{\Nrm^{\ssup{\mathrm{long}}}}

\newcommand{\Nmin}{{T_N}}
\newcommand{\Nmax}{N^{+}}
 
\definecolor{amethyst}{rgb}{0.6, 0.4, 0.8}

\title[ODLRO for the free Bose gas via Feynman--Kac formula]
{Off-diagonal long-range order\\ \medskip for the free Bose gas\\\medskip via the Feynman--Kac formula} 

\author{Wolfgang K\"onig$^1$}
\address{$^1$Weierstrass Institute Berlin, {\tt koenig@wias-berlin.de}}
\author{Quirin Vogel$^2$}
\address{$^2$Department of Mathematics, TU München, {\tt quirin.vogel@tum.de}}
\author{Alexander Zass$^3$}
\address{$^3$Weierstrass Institute Berlin, {\tt zass@wias-berlin.de}}

\usepackage[hang,flushmargin]{footmisc}

\providecommand{\noopsort}[1]{}
\mathtoolsset{showonlyrefs}

\begin{document}

\begin{abstract}
We consider the path-integral representation of the ideal Bose gas under various boundary conditions. We show that Bose--Einstein condensation occurs at the famous critical density threshold, by proving that its $1$-particle-reduced density matrix exhibits off-diagonal long-range order above that threshold, but not below. Our proofs are based on the well-known Feynman--Kac formula and a representation in  terms of a crucial Poisson point process. Furthermore, in the condensation regime, we derive a law of large numbers with strong concentration for the number of particles in short loops. In contrast to the situation for free boundary conditions, where the entire condensate sits in just one loop, for all other boundary conditions we obtain the limiting Poisson--Dirichlet distribution for the collection of the lengths of all long loops. 

Our proofs are new and purely probabilistic (apart from a standard eigenvalue expansion), using elementary tools like Markov's inequality, Poisson point processes, combinatorial formulas for cardinalities of particular partition sets, and asymptotics for random walks with Pareto-distributed steps.

\emph{\keywordsname:} Interacting many-particle systems, Bose--Einstein condensation, Brownian bridge, Feynman--Kac formula, long-range order, Poisson--Dirichlet distribution, partitions, random walks with heavy-tailed steps.
\end{abstract}
\maketitle

\section{Introduction and main results}\label{sec:intro}

\noindent In this paper, we consider a prominent and much-studied object, namely the free (i.e., ideal, non-interacting) Bose gas in a large box in the thermodynamic limit, and its famous Bose--Einstein condensation (BEC) phase transition. While in the physics literature the occurrence of this phase transition is considered a fact long-established, in the mathematics literature this appears to us much less clear. Furthermore, the well-known interpretation of BEC in terms of Brownian cycles (also called loops) has been stated a lot, but (as far as we are aware of) seems to have never been rigorously proved, at least not for all boundary conditions. 

The main purpose of this paper is to close this gap and provide proofs for the occurrence of the BEC phase transition using the well-known Feynman--Kac formula, and to draw rigorous consequences about the relation between the long Brownian cycles and the condensate. In doing so, we will also demonstrate the strength of using probabilistic methods -- like combinatorics, Poisson point processes and large deviations for random walks -- to study the free Bose gas.

The most popular definitions of the occurrence of BEC~\cite{PO56} are in terms of a macroscopic occupation of the zero Fourier mode and in terms of off-diagonal long-range order (ODLRO). The latter is defined as the existence of an eigenvalue of order $N$ for the $1$-particle-reduced density operator (1-prdm), and it is this definition that we will be relying on in this paper. It is well-known (e.g.,~\cite{Fey53,Gin70}) that the Feynman--Kac formula allows for a reformulation in terms of an ensemble of Brownian loops, each carrying a random number of particles (sometimes called \lq winding number\rq,~\cite{Uel06}). 
This is suggestive of a connection between BEC and the appearance of a macroscopic part in the system in long loops,~\cite{BDZ08}. Feynman himself suggested the portion of particles in long loops as a new order parameter, but warned that their physical relevance must be checked on a case-by-case basis. Hence, we consider it a useful goal to rigorously prove this relevance by deriving ODLRO using the Feynman--Kac formula, thereby rigorously showing that the mass of the condensate coincides with the amount of particles in the long cycles.

While most of the mathematical literature (see Section \ref{sec:lit} for a brief overview) focuses on the case of periodic boundary conditions, presumably because it gives a mathematical simplification, we wish to stress the relevance of considering different ones (e.g., zero, Neumann or mixed boundary conditions): not only because they may be more akin to the conditions used for experiments but also because, in the thermodynamic limit, they seem to influence statistical properties of the Bose gas that are relevant for BEC,~\cite{HKKS01,HKS02}. We will be considering both diffusive (e.g., zero Dirichlet and Neumann) and periodic boundary conditions, and comparing to the analogous formulas for the so-called free boundary conditions, where the influence of the boundary is ignored (note, however, that this does not come from a quantum mechanical model).

The organisation of the remainder of this section is as follows. In Section~\ref{sec:freeBosegas}, we introduce the free Bose gas. In Section~\ref{sec:ODLRO}, we present the main quantity of interest, the $1$-particle-reduced density matrix, and formulate our main result, Theorem \ref{thm:ODLRO}, on the occurrence of off-diagonal long-range order for large particle densities (the {\em supercritical case}) in dimensions $\geq 3$, and the non-occurrence for low densities (the {\em subcritical case}).
As a preparation for the proof, in Section \ref{sec:FKform}, we present the path-integral representation of the free Bose gas, and use it to rewrite the $1$-particle-reduced density matrix. In Section \ref{sec-PPP}, we introduce a fundamental and crucial Poisson point process of loop lengths, and present fine estimates for these lengths (Proposition~\ref{prop:PoissonDirichlet}). Some crucial assertions about the number of particles in short, respectively in long, loops for diffusive and periodic boundary conditions are formulated in Section~\ref{sec-loops}, in particular the distributional convergence of the collection of loop lengths towards the Poisson--Dirichlet distribution.  Finally, in Section~\ref{sec-freeenergy}, we compute explicit formulas for the limiting free energy. A literature survey is found in Section~\ref{sec:lit}.

Section~\ref{sec:proofODLRO} is dedicated to the proof of the occurrence of ODLRO in the supercritical regime, for periodic and diffusive boundary conditions, Section~\ref{sec:ODLRObc}, and for the free case, Section~\ref{sec:ODLROfree}. In Section~\ref{sec:nonODLRO}, we show that ODLRO does not occur in the subcritical case. In Section~\ref{sec-prooflongloops}, we prove the estimates on the distribution of the long loops. Remaining statements are proved in Appendix~\ref{sec:remaining}.

\subsection{The free Bose gas}\label{sec:freeBosegas}

For a fixed centred box $\L$ in $\R^d$, we consider an ideal bosonic system at positive temperature $1/\beta\in(0,+\infty)$ (working with the Boltzmann constant $k_{\rm B}=1$). 

More precisely, we introduce the $N$-particle Hamilton operator $\mathcal H_N^{\ssup{\L,{\bc}}}=-\frac 12 \sum_{i=1}^N\Delta_i^{\ssup{\L,{\bc}}}$ acting on the Hilbert space $L^2(\L^N)$ of square-integrable complex functions on $\L^N$, with kinetic energy given by the Laplace operator $\Delta^{\ssup{\L,{\bc}}}$ in $\L^N$ with a given boundary condition ${\bc}$. More precisely, we  consider the diffusion equation 
\begin{equation}
    \frac{\partial \psi}{\partial t}(x,t) = -\frac12\Delta\psi(x,t)\, ,\qquad t>0,x\in\L\, ,
\end{equation}
where either $\L$ is a torus (\emph{periodic} boundary conditions), or we impose what we call \emph{diffusive} boundary conditions, i.e.,
\begin{equation}
   \frac{\partial\psi}{\partial n}(x,t) = \bcd(x/L)\psi(x,t)\, \qquad t>0,x \in\partial\L,\mbox{ where }\L=L [-\smfrac 12,\smfrac 12]^d\, ,
\end{equation}
for a given function $\bcd\colon [-\frac 12,\frac 12]^d\setminus\rk{-\frac 12,\frac 12}^d \to[0,\infty]$, which is either continuously differentiable or $\equiv+\infty$ on the boundary of the centred unit box (in the latter case, this is to be interpreted as $\psi(x,t)=0$), see, e.g.,~\cite[page 373]{a1981operator}. We talk about \emph{Dirichlet} boundary conditions for $\bcd=+\infty$ and \emph{Neumann} boundary conditions for $\bcd=0$. Intermediary values of $\bcd$ correspond to semi-reflecting, semi-absorbing boundary conditions.

In this work, for the purpose of having a uniform presentation, fixing a boundary condition `bc' will mean that we choose a function $\bcd$ or study the problem on the torus. We write $\bc\in\{\diff,\per\}$.

The partition function of the system is given by the symmetrised trace
\begin{equation}
    Z_N^{\ssup{\L,{\bc}}}(\beta)={\rm Tr}\Big(\Pi_+\circ{\rm e}^{-\beta \mathcal H_N^{\ssup{\L,{\bc}}}}\Big),\qquad N\in\N, \beta\in(0\,\infty)\, ,
\end{equation}
see \cite[Eq. 2.21]{Gin70}, where $\Pi_+\colon L^2(\L^N)\to L^2_{\rm{sym}}(\L^N)$ is the symmetrisation operator $\Pi_+ f(x_1,\dots,x_N)=\frac 1{N!}\sum_{\sigma\in\Sym_N}f(x_{\sigma_1},\dots,x_{\sigma_N})$, $\Sym_N$ denotes the set of all  permutations of $1,\dots,N$, and $\circ$ is the composition of operators.

We study this trace in the thermodynamic limit, and will therefore choose $\L$ to be the centred box $\L_N$ of volume $N/\rho$, where $\rho\in(0,\infty)$ is the {\em particle density}. Since $\beta>0$ remains fixed throughout this paper, we omit this dependency from the notations. It is known that the {\em free energy per volume},
\begin{equation}\label{freeenergy}
    {\rm f}(\rho)=-\frac 1\beta \lim_{N\to\infty}\frac 1{|\L_N|}\log  Z_N^{\ssup{\L_N,{\bc}}}(\beta)\, ,
\end{equation}
exists and does not depend on the boundary condition. In fact, it exhibits a phase transition in dimensions $d\geq3$ at the critical particle density threshold
\begin{equation}\label{Rhoc}
\rho_{\rm c}=\sum_{k\in\N}(2\pi k\beta)^{-d/2}=(2\pi\beta)^{-d/2}\zeta(d/2)\begin{cases}
=\infty&\mbox{in }d\leq 2,\\
<\infty&\mbox{in }d\geq 3\, ,
\end{cases} 
\end{equation}
where $\zeta(a)=\sum_{k\in\N}k^{-a}$ denotes the Riemann zeta function. Indeed, ${\rm f}$ is analytic in $(0,\rho_{\rm c})$ and constant in $[\rho_{\rm c},\infty)$; see Section~\ref{sec-freeenergy} for more details. Clearly, $\rhoc$ does not depend on the boundary conditions.

\subsection{Main results: ODLRO and limiting distribution of loop lengths}\label{sec:ODLRO}

We would like to identify the phase transition seen in the non-analyticity of the free energy as the famous {\em Bose--Einstein condensation} phase transition. For this, we need to introduce the {\em $1$-particle-reduced density matrix}, i.e., the kernel $\gamma_N^{\ssup{\L,{\bc}}}\colon \L\times\L\to[0,\infty)$ of the following partial trace over $N-1$ variables:
\begin{equation}
   \Gamma_N^{\ssup{\L,{\bc}}}= \frac N{Z_N^{\ssup{\L,{\bc}}}} {\rm Tr}_{N-1}\Big( \Pi_+\circ{\rm e}^{-\beta \mathcal H_N^{\ssup{\L,{\bc}}}}\Big)\, .
\end{equation}
The operator $\Gamma_N^{\ssup{\L,{\bc}}}$ acts on $L^2(\L)$. Its principal (i.e., largest) eigenvalue is given by
\begin{equation}
    \sigma_N^{\ssup{\L,{\bc}}}=\sup_{f\in L^2(\L)\colon \|f\|_{L^2(\L)}=1}\langle f,\Gamma_{N}^{\ssup{\L,{\rm bc}}} (f)\rangle.
\end{equation}
We say that the system exhibits {\em off-diagonal long-range  order} if $\sigma_N^{\ssup{\L_N,{\bc}}}\asymp N$ in the thermodynamic limit at some density $\rho$.
The BEC phase transition at the critical density $\rhoc$ is then mathematically defined as the occurrence of ODLRO for $\rho>\rhoc$, but not for $\rho<\rho_{\rm c}$.

The main purpose of this paper is to use the Feynman--Kac formula to prove that this property holds for the above-mentioned boundary conditions. The novelty is twofold: firstly, we prove ODLRO not only for periodic, but also for \textit{all} boundary conditions (including Dirichlet and Neumann); this was -- to the best of our knowledge -- missing in the literature. Secondly, our proof is intrinsically probabilistic, as it uses the Feynman--Kac formula and works in the setting of an ensemble of many Brownian loops with long and short lengths. Indeed, we also prove ODLRO for the mathematical model of free boundary conditions. See also the discussion in Section \ref{sec:lit}.
 
We denote by $U=[-\frac 12,\frac 12]^d$ the closed centred unit box, and by $\phi_1^{\ssup{\bc}}\colon U\to[0,\infty)$ the $L^2(U)$-normalised eigenfunction corresponding to the smallest eigenvalue of $-\frac12 \Delta^{\ssup{U,{\bc}}}$. For example, $\phi_1^{\ssup{\free}}(x)=\phi_1^{\ssup{\per}}(x)=\phi_1^{\ssup{\neu}}(x)=1$ and $\phi_1^{\ssup{\dir}}(x)= 2^{d/2}\prod_{i=1}^d\cos(\pi x_i)$, for $x=(x_1,\dots,x_d)$. We can now state our main result. 

\begin{theorem}[ODLRO for the free Bose gas with boundary conditions]\label{thm:ODLRO} 
Fix $d\in\N$, let $\L_N$ be the centred box $L_N  U$  with volume $L_N^d=N/\rho$ where $U=[-\frac 12,\frac 12]^d$. Fix any periodic or diffusive boundary condition $\bc$, and $\beta,\rho\in(0,\infty)$. Then the following hold:

\begin{enumerate}
\item[(i)] {\em (Supercritical regime: $\rho>\rhoc$.)} Assume that $d\geq 3 $ and $\rho>\rhoc$, then the kernel satisfies, uniformly in $x,y \in L_N$, as $N\to\infty$,
\begin{equation}\label{eq:ODLRO_sup}
   \gammabc{\bc}(x,y)= \rk{\rho-\rhoc+o(1)} \phi_1^\ssup{\bc}(\smfrac x{L_N})\phi_1^\ssup{\bc}(\smfrac y{L_N})+ \psi(\abs{x-y})+o(1),
\end{equation}
with some function $\psi\colon(0,\infty)\to(0,\infty)$ satisfying $\psi(r)\leq C r^{2-d}$ as $r\to\infty$ for some $C>0$. As a consequence,
\begin{equation}
\sigma_N^{\ssup{\L_N,{\bc}}}\sim \rk{\rho-\rhoc}\abs{\L_N},\qquad N\to\infty.
\end{equation}

\item[(ii)] {\em (Subcritical regime: $\rho\leq \rhoc$.)} For some $c>0$ and all $x,y\in \L_N$, as $N\to\infty$,
\begin{equation}\label{eq:ODLRO_sub}
    \gammabc{\bc}(x,y)=\Ocal\rk{\ex^{-c\abs{x-y}}}\, .
\end{equation}
As a consequence,
\begin{equation}
\sigma_N^{\ssup{\L_N,{\bc}}}\leq  \Ocal(1).
\end{equation}
\end{enumerate}
\end{theorem}
Note that the same results hold true for free (open) boundary conditions (see Proposition \ref{prop:ODLRO_free}), even though it is not a quantum mechanical model; see the definition of $\gammabc{\free}$ in \eqref{eq:rdm_bc} below for $\bc=\free$. 

The proof of the theorem is in Section \ref{sec:proofODLRO} for the supercritical regime, and in Section \ref{sec:nonODLRO} for the subcritical one.
The main message of Theorem~\ref{thm:ODLRO} is that the 1-prdm is asymptotically equal to the `overshoot', the mass of the condensate, times the tensor product of the eigenfunction associated to the smallest eigenvalue, up to an additive error that is small away from the diagonal.

\subsection{Expansion in terms of the Feynman--Kac formula}\label{sec:FKform}

It has been known for more than 50 years that the trace of the density matrix admits a path-integral representation in terms of Brownian bridges (see, e.g.,~\cite{Gin70}). This representation is sometimes called the \emph{Feynman--Kac formula}; it appears in \eqref{FKform} below. The following lemma is the starting point of our analysis, and we consider the objects that it introduces as important mathematical quantities that are worth being studied in their own right; see our additional results on the behaviour of the loop lengths in Proposition \ref{prop:PoissonDirichlet} below. 
Sh
Fix a centred box $\L\subset\R^d$. We denote by $g_{r\beta}^{\ssup{\L,{\bc}}}(x,y)$ the kernel of the fundamental solution to the Cauchy problem for the heat equation $\frac12 \Delta^{\ssup{\L,{\bc}}} g=0$ in $\L$ with boundary condition ${\bc}$. By
\begin{equation}\label{bkdef}
    \q_{k}^{\ssup{\L,{\bc}}}=\int_\L g_{k\beta}^{\ssup{\L,{\bc}}}(x,x)\,\d x,\qquad k\in\N\, ,
\end{equation}
we denote the trace of the convolution operator with kernel $g_{k\beta}^{\ssup{\L,{\bc}}}$. We write
\begin{equation}\label{partitions}
    \Pfrak_N=\Big\{\partition=(\partition_k)_{k\in\N}\in\N_0^\N\colon \sum_{k\in\N}k\partition_k=N\Big\}
\end{equation}
for the set of partitions of $N$.

We are now ready to present the key representations of the partition function and $1$-prdm. The proof is in Appendix \ref{sec:remaining}.
\begin{lemma}[FK-representations]\label{lem:FKrepgamma} 
Fix a box $\L\subset\R^d$ and any periodic or diffusive boundary condition $\bc$. Then, for any $N\in\N$ and $\beta\in(0,\infty)$,
\begin{equation}\label{eq:FK_Z}
    Z_N^{\ssup{\L,{\bc}}}=\sum_{\partition\in\Pfrak_N}\prod_{k\in\N}
    \frac{(\q_k^{\ssup{\L,{\bc}}})^{\partition_k}}{k^{\partition_k}\partition_k!}\, ,
\end{equation}
and
\begin{equation}\label{eq:firstformula}
    \gamma_N^{\ssup{\L,{\bc}}}(x,y)=\sum_{r=1}^{N}g^{\ssup{\L,{\bc}}}_{r\beta}(x,y)\frac{Z^{\ssup{\L,{\bc}}}_{N-r}}{Z^{\ssup{\L,{\bc}}}_N},\qquad x,y\in\L\, .
\end{equation}
\end{lemma}

\subsection{A crucial Poisson point process}\label{sec-PPP}

In \cite{ACK11}, the observation was made that the ensemble of Brownian loops in the Feynman--Kac formula can be conceived as a Poisson soup of sites, appended with a mark that is the Brownian loop starting and ending at this site. As our second main reformulation step, we will rewrite the quotient of partition functions in \eqref{eq:firstformula} in terms of probabilities with respect to such a Poisson point process (PPP). Since we are not interested here in the entire loops, but only in their lengths, the mark of a Poisson point will be just the length of the appended loop. This leads to a reduced version of that PPP, which we might see as a kind of {\em grand-canonical point process} on $\N$.

Fix a centred box $\L\subset \R^d$, boundary condition ${\bc}$, and $N\in\N$. Let $\PPP^{\ssup{{\bc,N}}}_{\L}$ be the distribution of a Poisson point process $(\omega_i)_i$ on $\N$ with intensity measure given by
\begin{equation}\label{intensitymeasure}
    \nu^\ssup{{\bc,N}}_{\L} = \sum_{k=1}^{N} \frac{1}{k} \q_{k}^{\ssup{\L,{\bc}}}\delta_k\, .
\end{equation} 
The total mass of $\nu_{\L}^{\ssup{\bc,N}}$ is equal to $|\L| \presbcN_\L$, where
\begin{equation*}
  \presbcN_\L = \frac{1}{\al}\sum_{k= 1}^{N}\frac 
  1 k\, \q_{k}^{\ssup{\L,{\bc}}}
\end{equation*}
is the pressure in $\L$.

The points $\omega$ of the PPP sitting at $k\in\N$ are called the \emph{loops of length $k$}. The number $X_k$ of such points is a Poisson-distributed random variable with parameter $\frac{1}{k} \q_{k}^{\ssup{\L,{\bc}}}$, and the family $(X_k)_{k\in\N}$ is independent. The number $\Nop_\L(\eta):=\sum_{k=1}^N X_k$ of Poisson points in $\L$ is then Poisson-distributed with parameter $\presbcN_\L|\L|$. Since $X_k$ counts the number of loops of length $k$, we say that the total number of \emph{particles} of $\eta$ that belong to loops starting in $\L$ is given by $\Partn_{\L}(\eta) := \sum_{k=1}^N k X_k$. 
We denote by $\poi{\lambda}$ the Poisson distribution with parameter $\lambda>0$.

The following expresses the partition function of the bosonic system in terms of the above PPP.
\begin{lemma}\label{lem-PoissonRepr}
For any centred box $\L\subset\R^d$, $N\in\N$, $\beta\in(0,\infty)$, and any diffusive or periodic $\bc$, we have
    \begin{equation}\label{ZPPPrepr}
        Z_N^{\ssup{\L,\bc}} = \ex^{\abs{\L} \presbcN_\L}\PPP^{\ssup{\bc,N}}_{\L}(\Partn_{\L} = N).
    \end{equation}
\end{lemma}

\begin{proof}We drop all super-indices from the notation in the proof. We start from \eqref{eq:FK_Z}, that is
    \begin{equation}
    \begin{split}
        Z_N&=\sum_{\partition\in\Pfrak_N}\prod_{k\geq 1} \frac{(\q_k/k)^{\partition_k}}{\partition_k !}= \ex^{\sum_{k\geq 1}\frac 1k \q_k}\sum_{\partition\in\Pfrak_N} \prod_{k\in\N} \P(X_k = \partition_k).
    \end{split}
    \end{equation}
    From the independence of the $X_k$'s, since $\sum_k k\partition_k = N$ and $\Partn_{\L} = \sum_k kX_k$, the claim follows. Indeed,
    \begin{equation}
    \begin{split}
        \PPP_{\L}(\Partn_{\L} = N) &= \sum_{\partition\in\Pfrak_N} \PPP_{\L}(\Partn_{\L} = N | X_k = \partition_k\ \forall k)\prod_k \P(X_k = \partition_k)\\
        &= \sum_{\partition\in\Pfrak_N} \prod_k \P(X_k = \partition_k)\, .
    \end{split}
    \end{equation}
\end{proof}

Combining this with Lemma~\ref{lem:FKrepgamma} yields the following representation.

\begin{cor}\label{cor-FKPPPrepr}
For any centred box $\L\subset\R^d$, $N\in\N$, any diffusive or periodic boundary condition $\bc$, and for any $x,y\in\L$, we have
\begin{equation}\label{eq:rdm_bc}
  \gamma_N^{\ssup{\L,{\bc}}}(x,y)=  \sum_{r=1}^{N}g_{r\beta}^{\ssup{\L,{\bc}}}(x,y)
  \frac{\PPP_{\L}^{\ssup{{\bc,N}}}\left(\Partn_{\L}=N-r\right)}{\PPP_{\L}^{\ssup{{\bc,N}}}\left(\Partn_{\L}=N\right)}.
\end{equation}
\end{cor}
This will be the starting point of our analysis.

The formula in~\eqref{eq:rdm_bc} allows us to introduce what we call {\em free boundary conditions}. In place of the heat-equation solution $g_{k\beta}^{\ssup{ \L,\bc}}$ with some boundary condition $\bc$, we consider the solution with {\em free boundary condition}, i.e., the standard kernel $g_{k\beta}^{\ssup{\free}}(x,y)=(2\pi\beta k)^{-d/2}\exp\{-\frac 12 |x-y|^2/2\beta k\}$. Substituting this in \eqref{eq:rdm_bc} is mathematically sound and simplifies and streamlines the formulas and asymptotics; however, one needs to keep in mind that the kernel $\gamma_N^{\ssup{\free}}$ does not come from any quantum model and is not the trace of any symmetrised operator. Indeed, the restriction of $g_{k\beta}^{\ssup{\free}}$ to a box $\L$ does not satisfy the Chapman--Kolmogorov equations, hence Lemma~\ref{lem:FKrepgamma} does not apply to the free-bc case. Nevertheless, we include the free-bc case into our analysis and tacitly introduce all objects like $\q_{k}^{\ssup{\L,{\bc}}}$, $\nu^\ssup{{\bc}}_{\L}$, the Poisson process, and so on also for $\bc=\free$, in which case, the above definitions are for $N=+\infty$.

The analogous result of Theorem \ref{thm:ODLRO} holds:
\begin{proposition}[ODLRO for the free Bose gas with free boundary condition]\label{prop:ODLRO_free}
    For $d\geq 3$, let $\L_N$ be the centred box $L_N  U$  with volume $L_N^d=N/\rho$ where $U=[-\frac 12,\frac 12]^d$. Fix $\beta,\rho\in(0,\infty)$. Then analogous asymptotics of $\gammabc{\free}$ as in \eqref{eq:ODLRO_sup}, respectively in \eqref{eq:ODLRO_sub}, hold, see Proposition \ref{prop:ODLRO_sup_free} for the precise statement. 
\end{proposition}

Our proof of Proposition~\ref{prop:ODLRO_free} is in Section~\ref{sec:ODLROfree}.
 It is based on an extended (\lq spatial\rq) version of our Poissonian representation and is pretty different from the proof of Theorem~\ref{thm:ODLRO}.

\subsection{Particle numbers in short and long loops}\label{sec-loops}

One of the purposes of the present paper is to analyse the behaviour of the statistics of the loop lenghts and the number of particles in the loops in the Bose gas. These quantities play a decisive role in the understanding of the Bose gas and are since decades under interest of many researchers. They are interesting mathematical objects, worth to be studied on their own. After having introduced the crucial PPP in Section~\ref{sec-PPP}, we have now a mathematical frame to talk about these quantities. Indeed, in the frame of the Feynman--Kac representation of Section~\ref{sec:FKform}, we  introduce a natural random variable with values in $\Pfrak_N$ with distribution
\begin{equation}
 \Pfrak_N\ni   (\partition_r)_{r=1}^N= \partition\mapsto \frac{1}{Z_N}\prod_{k=1}^N\frac{\q_k^{\partition_k}}{k^{\partition_k}\partition_k!}\, .
\end{equation}
Thanks to the PPP formulation from Section~\ref{sec-PPP}, its distribution can be written as
\begin{equation}
    \PPP^{\ssup{\bc,N}}_\L\big((X_r)_r\in\cdot\, \vert \Partn_\L = N\big)\, .
\end{equation}
Hence, we can now talk about particle numbers in certain loops in the framework of the process $(X_r)_{r=1}^N$.  Denote the number of particles in loops of length between $l_1$ and $l_2$ by
\begin{equation}\label{particlesinloops}
{\rm N}_\L^{\sqsup{l_1,l_2}}=\sum_{r=l_1}^{l_2}r X_r,\qquad l_1,l_2\in\N_0\, .
\end{equation}
Since Feynman's~\cite{Fey53} proposal to conceive the lengths of the longest loops as an order parameter, a high interest lies on the question of whether or not macroscopically long loops occur, whether this coincides with the occurrence of BEC, and whether or not the long loops themselves have something to do with the condensate. Most of the answers that are given for the free gas in the literature apply only to free boundary conditions. However, as we will see, the behaviour of the lengths of the long loops differs significantly between the gas with diffusive or periodic boundary conditions, as well as the model with free boundary conditions. In fact, while for the latter there is only one scale of the length distribution, as their are not influenced by the box, for diffusive or periodic boundary conditions, in the presence of a macroscopic box, the distribution of the loop lengths is significantly different for small and large lengths. The critical threshold lies close to $L_N^2$, the square of the diameter of the box $\L_N$. For example, for Dirichlet zero boundary condition, the long loops are probabilistically suppressed in a stretched-exponential way, while their probability under free boundary condition has polynomial tails. Actually, we will put the threshold between short and long loops at
\begin{equation}\label{Nminus}
    \Nmin= \lceil L_N^2\log^{1/2}(N)\rceil,\qquad N\in\N\, .
\end{equation}
Note that $\Nmin$ is almost equal to the mixing time ($L_N^2$) of Brownian motion conditioned to stay in $\L_N$. The reason for this choice of threshold is that typical loops shorter than $\Nmin$ behave then like Brownian motions, whereas longer loops follow the stationary distribution (the $\log$-correction is for technical convenience). Loops with length $\leq \Nmin$ are called \emph{short}, the others \emph{long}. We denote the number of particles in short and in long loops, respectively, by
\begin{equation}
    \Ns_\L={\rm N}_\L^{\sqsup{1,\Nmin}}=\sum_{k=1}^{\Nmin} k X_k\qquad\mbox{and }\qquad \Nl_\L={\rm N}_\L^{\sqsup{\Nmin+1,N}}=\sum_{k > \Nmin} k X_k\, .
\end{equation}
Our first result concerning the number of particles in short loops in the thermodynamic limit is the following. We see that $\frac 1 {|\L_N |} {\Ns_{\L_N}}$ converges to $\rhoc$, the critical threshold defined in \eqref{Rhoc}. Actually, the deviations of this quantity from its expectation are rather small: the probability of such a deviation is stretched-exponentially small in $N$. To demonstrate this, we divide the short loops into finite length and the remainder. For $R\in\N$, we introduce $\rhoc^{\ssup R}=(2\pi \beta)^{-d/2}\sum_{k=1}^R k^{-d/2}$, which converges to $\rhoc$ as $R\to\infty$.

\begin{proposition}\label{prop-shortloopsasy}
Fix any boundary condition $\bc\in\gk{\diff,\per, \free}$. Fix $d\geq 3$ and $\rho\in(\rhoc,\infty)$, and consider the centred box $\L_N$ with volume $N/\rho$. Then, for any $\e>0$ and any $R\in\N$, there exists $C_{\e,R}>0$ such that, for any large $N$,
    \begin{equation}
        \PPP^{\ssup{\bc,N}}_{\L_N}\rk{\abs{\frac{1}{\alN}\Partn_{\L_N}^{\sqsup{1,R}}-\rhoc^{\ssup{R}}}>\e\ \Big\vert\ \Partn_{\L_N}=N } \leq \ex^{-C_{\e,R}\alN}\, .
    \end{equation}
    Moreover, for any (large) $\kappa>0$ and for any (small) $\e>0$, for diffusive or periodic boundary conditions, for any large $N$,
    \begin{equation}\label{asymiddleshortloops}
        \PPP^{\ssup{\bc,N}}_{\L_N}\rk{\abs{\frac{1}{\alN}\Partn_{\L_N}^{\sqsup{R+1,\Nmin}}-(\rhoc-\rhoc^{\ssup{R}})}>\e\ \Big\vert\ \Partn_{\L_N}=N } \leq \ex^{-\kappa\alN^{1-\frac2d}}\, ,
    \end{equation}
    while for free boundary conditions, \eqref{asymiddleshortloops} is true for  $\Partn_{\L_N}^{\sqsup{R+1,\Nmin}}$ replaced with $\Partn_{\L_N}^{\sqsup{R+1,N}}$.
\end{proposition}

That is, the particle number in loops of length in $(R_N,\Nmin]$, respectively in $(R_N,N]$, is $o(|\L_N|)$ with very high probability. The proof of Proposition~\ref{prop-shortloopsasy} follows from Propositions~\ref{prop:ShortLoops} and \ref{prop:MeanShort}, together with Lemma~\ref{PropsotionUpperBoundTotal}.

Now let us turn to the particle numbers in long loops. The reader will see that, by combining \eqref{eq:expSubcrit} and \eqref{eq:sumSubcrit}, one gets that $\EPP_{\L_N}^{\ssup{\bc,N}}[\frac{1}{\alN}\Partn_{\L_N}]\rightarrow \rho\wedge\rhoc$, hence, $\Nl_{\L_N}$ is $o(|\L_N|)$ under the free measure $\PPP_{\L_N}^{\ssup{\bc,N}}$. However, we are interested, for $\rho>\rhoc$, in the behaviour under the conditioning on $\Partn_{\L_N}$ being equal to $N=\rho |\L_N|$, i.e., on having many more particles than in the `usual' behaviour. From Proposition~\ref{prop-shortloopsasy}, we see that $\Ns_{\L_N}$ will hardly contribute to this extremal event, so $\Nl_{\L_N}$ will contribute practically everything. It turns out below that the probabilistically-cheapest way to do so is to create a finite number of macroscopically-large loops that altogether make up for the requested amount $(\rho-\rhoc)|\L_N|$. However, the precise way to do this differs between diffusive or periodic boundary conditions, and free boundary conditions.

For a realisation $(X_r)_{r=1}^N$ of the Poisson point process, let $L^{\ssup N}_1\geq L^{\ssup N}_2 \geq \dots\geq 0$ denote the lengths of all the loops in the process, ordered according to their size. Hence, $L^{\ssup N}_1=\max\{r\colon X_r>0\}=r^*>0$, and $L^{\ssup N}_1 =L^{\ssup N}_2=\dots = L^{\ssup N}_{X_{r^*}}>L^{\ssup N}_{X_{r^*}+1} $ and so on. In particular, $\sum_{i\in\N}L^{\ssup N}_i= \sum_{r=1}^N r X_r$. On the event $\{{\rm N}_\L=N\}$, the sequence $(L^{\ssup N}_i)_{i\in\N}$ forms a partition of $N$.

Let us recall that the Poisson--Dirichlet distribution with parameters $0$ and $1$ (denoted $\textsf{PD}_1$) is given as the joint distribution of the random variables $(Y_n\prod_{k=1}^{n-1} (1-Y_k))_{n\in\N}$, where $(Y_n)_{n\in\N}$ is an i.i.d.~sequence of ${\rm Beta}(1,1)$-distributed random variables (i.e., uniformly over $[0,1]$ distributed). Note that the sum of the elements of a $\textsf{PD}_1$-distributed sequence is equal to one, i.e., this distribution is in fact a random partition. It is well-known in asymptotics for random permutations, as the joint distribution of the lengths of all the cycles of a uniformly picked random permutation of $1,\dots,N$, ordered according to their sizes and normalized by a factor $1/N$, converges weakly to $\textsf{PD}_1$.

We obtain the following:
\begin{proposition}[Lengths of long loops]\label{prop:PoissonDirichlet}
Fix any $\bc\in \gk{\diff,\per}$.  Fix $\rho\in(\rhoc,\infty)$ and consider the centred box $\L_N$ with volume $N/\rho$. Then, under $\PPP_\L^{\ssup{\bc,N}}$, conditional on $\{\Partn_\L=N\}$, as $N\to \infty$,
\begin{equation}
   \frac {(L^{\ssup N}_i)_{i\in\N}}{|\L_N|(\rho-\rhoc)}\Longrightarrow
    \textsf{PD}_1.
\end{equation}
\end{proposition}
\begin{remark}
The above proposition states that ODLRO occurs for the same particle densities as those for which macroscopic loop(s) occur, and that the number of particles in the long loops is asymptotically equal to the mass of the condensate. However, the precise distribution of the macroscopic loop lengths depends on the boundary condition. 

Indeed, it is known (see, e.g.,~\cite[Corollary 5.4]{VOGEL2023104227}) that for free boundary conditions, with high probability under the above conditioning, there exists a single macroscopic loop with length $\sim (\rho-\rhoc)|\L_N|$, and the length of the second-longest is $o(|\L_N|)$. In fact, the loop lengths starting from the second follow the limiting order-statistics for i.i.d.~Pareto-distributed variables on a scale $o(|\L_N|)$. In particular, we do not observe the Poisson--Dirichlet distribution in the limit.
Furthermore, macroscopic loops can be described (after taking the limit) by random interlacements. We note that random interlacements can be generated from a single large loop, multiple macroscopic loops, or a diverging number of mesoscopic loops: when proving a convergence to random interlacements, information about the lengths of the underlying loops is lost, see~\cite{bouchot2024}.
\end{remark}

Our proof of Proposition~\ref{prop:PoissonDirichlet} is in Section~\ref{sec-prooflongloops}. It is based on our Poissonian representation and differs significantly from other proofs in the literature for some special cases.

\subsection{The free energy}\label{sec-freeenergy}
From our results, we may easily identify the free energy for all boundary conditions.
For the proof of this statement, see Appendix~\ref{sec:remaining}.
\begin{proposition}[Identification of the free energy]\label{PropFreeEnergy}
Fix $N\in\N$ and $\rho\in(0,\infty)$ and consider the centred box $\L_N$ of volume $N/\rho$. Fix any boundary condition $\bc\in \gk{\diff,\per, \free}$. Then, for $\rho\not=\rhoc$ 
\begin{equation}\label{freeenergyident}
    {\rm f}(\rho)= -\frac 1\beta \lim_{N\to\infty}\frac 1{|\L_N|}\log  Z_N^{\ssup{\L_N,{\bc}}}=
    \mu(\rho)\rho-\frac 1 \beta \sum_{j\in\N}\frac{\ex^{\beta \mu(\rho) j}}{j \rk{2\pi \beta j}^{d/2}},
\end{equation}
where $\mu(\rho)$ is defined by
\begin{equation}\label{mudef}
\rho\wedge \rhoc=\sum_{k\in\N}\ex^{\beta \mu(\rho) k}(2\pi \beta k)^{-d/2}.
\end{equation}
\end{proposition}
Note that the right-hand side of \eqref{freeenergyident} is continuous in $\rho\in\R$ and equal to $-\beta^{-1-d/2}\zeta(d/2)$ for $\rho>\rhoc$.

\subsection{Literature remarks}\label{sec:lit}

The occurrence of BEC, i.e., that a positive fraction of bosons are assembled in the lowest energy state, seems to be an established fact within the physics community, and indeed the physical literature on the free Bose gas with various boundary conditions is extensive, see, e.g.~\cite{Lon38,ZUK77,BL82,vdB83,WMH11,KPS20}.

BEC is usually characterised, for the ideal gas, in terms of the phenomenon of non-trivial occupation of the zero Fourier mode, i.e., that a positive fraction of the number of particles occupies the state of zero momentum, as this can be expressed as the Fourier transform $\widehat\gamma\Ssup{\L,{\bc}}_N(0)$, see~\cite{PO56}. In this paper, the authors also provided ODLRO as alternative definition of BEC, arguing that, for periodic boundary conditions, it allowed to extend the Fourier mode criterion to interacting gases. In~\cite{Gir65}, it was noted that this reasoning could also be extended to other boundary conditions, see also~\cite{CK07} and~\cite{BDZ08}. 

Furthermore, we were not able to find a rigorous proof (at least not for all boundary conditions) of the occurrence of ODLRO (other than via the zero Fourier mode ansatz) for large enough densities in the mathematical literature. Even in the probabilistic works that are often cited for proving BEC, e.g.,~\cite{Suto93,Suto02,Uel06,Ital05}, the authors only work with periodic boundary conditions. Furthermore, often the notion of BEC used there is either not based on ODLRO nor on zero Fourier mode, but on the occurrence of long cycles in the Feynman--Kac formula in certain senses; or additional scalings are introduced (e.g. mean-field cycle weights) which do not follow from the quantum mechanical set-up. This approach is known at least since~\cite{Gin70}, and it was discussed also in~\cite{Uel06} and~\cite{CK07} for free boundary conditions, without proofs. In~\cite{Ital05}, among other interesting assertions, the phenomenon of ODLRO for the free Bose gas with an additional exponential tilting, leading to a kind of mean-field model, was proved for periodic boundary conditions, but the exponential tilting was simplifying the proof and is not easily overcome. Moreover, the distribution of long loops was investigated in~\cite{Betz_Ueltschi_2008,Betz_Ueltschi_2011,Betz_Ueltschi_Velenik}, for models of weighted spatial permutations that can be related to the Bose gas with periodic boundary conditions. We would also like to mention the recent study of the infinite volume limit in the free case~\cite{armendariz2021gaussian} and thermodynamic limit result~\cite{VOGEL2023104227} relating the supercritical Bose gas to random interlacements.

Hence, even though quite a number of authors showed a great interest in the Feynman--Kac approach to the Bose gas, we were not able to find in the literature a probabilistic proof of ODLRO for all boundary conditions. In the present paper, we therefore provide a proof for ODLRO for all boundary conditions, using the Feynman--Kac formula, thereby closing what we consider a gap in the literature. Furthermore, we provide rigorous assertions about the number of particles in long and short loops.


%

\section{Proof of ODLRO in the supercritical regime}\label{sec:proofODLRO}
In this section, we present the proof of the occurrence of off-diagonal long-range order (ODLRO) above the critical threshold $\rhoc$, for all boundary conditions, that is, the proof  of Theorem \ref{thm:ODLRO}(i). Hence, we assume that $d\geq 3$ and fix $\rho\in(\rhoc,\infty)$ for the remainder of this section. 

A survey on this section is as follows. In Section~\ref{sec:ODLRObc} we formulate, for diffusive or periodic boundary condition, our main statement on the asymptotics of the kernel $\gammabc{\bc}(x,y)$, the proof of which is finished in Section~\ref{sec-proofProp2.1}. A crucial tool box for bounding the solution $g_{k\beta}^{\ssup{\L,{\rm bc}}}(x,y)$ to the heat equation and its trace $\q_k^{\ssup{\L,{\rm bc}}}$ is provided in Section~\ref{sec-EVexp}, based on spectral theory and eigenvalue expansions. In Section~\ref{sec-smalloops} and ~\ref{sec-longloops}, respectively, we derive some fine asymptotics for the number of particles in short, respectively in long, loops. While the method used in Section~\ref{sec-smalloops} is based on standard probability estimates, the material on Section~\ref{sec-longloops} constitutes the core of novelty that we needed to derive in this paper. Section~\ref{sec-proofODLROnonfree} then finishes the proof of ODLRO for diffusive and periodic boundary condition (i.e., Theorem \ref{thm:ODLRO}(i)), based on all the preceding material of this section. Finally, Section~\ref{sec:ODLROfree} proves Proposition~\ref{prop:ODLRO_free} (i.e., ODLRO for free boundary condition), using ad-hoc methods.

\subsection{Asymptotics of the kernel for diffusive and periodic boundary conditions}\label{sec:ODLRObc}

In this section we formulate Proposition~\ref{prop:ODLRO_sup} about lower and upper bounds for the kernel for diffusive and periodic boundary conditions and give a survey on its proof. We drop the superscript \lq  bc\rq.
The main result of this section is the following. Recall that $U=[\frac 12,\frac 12]^d$ is the closed centred unit box in $\R^d$.

\begin{proposition}\label{prop:ODLRO_sup}
For any diffusive or periodic boundary condition $\bc$, fix $\rho\in (\rhoc,\infty)$ and consider the centred box $\L_N$ of volume $L_N^d=N/\rho$ for $N\in\N$. Then there is $c\in(0,\infty)$ such that, for any $\kappa,M>0$, for all $x,y\in\L_N$, in the limit $N\to\infty$, 
\begin{multline}
  (1+o(1)) (\rho-\rho_{\rm c})\rk{\phi_1(\smfrac{x}{L_N})\phi_1(\smfrac{y}{L_N})}\leq \gammabc{\bc}(x,y)  \\
    \leq (1+o(1)) (\rho-\rho_{\rm c})\Big(\phi_1(\smfrac{x}{L_N})\phi_1(\smfrac{y}{L_N})+\ex^{-M c}\Big)+C_M \psi(|x-y|)+ \ex^{-\kappa L^{d-2}}\, ,
\end{multline}
where $C_M\in(0,\infty)$ depends only on $M$, and $\psi\colon (0,\infty)\to(0,\infty)$ satisfies $\psi(r)\leq \Ocal(r^{2-d})$ as $r\to\infty$.
\end{proposition}

The proof is in Section~\ref{sec-proofProp2.1}; here is a summary. Our starting point is \eqref{eq:rdm_bc} in Corollary~\ref{cor-FKPPPrepr}. First, we will use a standard eigenvalue expansion to approximate 
\begin{equation}\label{gexpansion}
    g^{\ssup \L}_{\beta r}(x,y)\sim \frac1{|\L|}\phi_1(\smfrac xL) \phi_1(\smfrac yL)\ex^{-\lambda_1\beta r \al^{-\frac 2d }}\, ,
\end{equation}
where $\lambda_1$ is the smallest eigenvalue of $-\frac12 \Delta^{\ssup{U}}$ in the unit box $U=[-\frac 12,\frac 12]^d$, and $\phi_1$ is the associated $L^2$-normalised eigenfunction. In case of Dirichlet, periodic or Neumann boundary conditions, we have,
\begin{equation}\label{princeigen}
\begin{gathered}
    \lambda_1^{\ssup{\dir}}=\frac{\pi^2}2 d,\quad \lambda_1^{\ssup{\per}}=0=\lambda_1^{\ssup{\neu}}\, ,\\
    \phi_1^{\ssup{\dir}}(x)= 2^{d/2} \prod_{i=1}^d \cos\big(\pi x_i\big),\quad \phi_1^{\ssup{\per}}(x)=1=\phi_1^{\ssup {\neu}}(x)\, .
\end{gathered}
\end{equation}
Furthermore, we decompose the total number of particles, 
$ {\rm N}_{\L_N}$, into the numbers of particles 
${\rm N}^{\ssup{\rm short}}_{\L_N}$ and ${\rm N}^{\ssup{\rm long}}_{\L_N}$ of short and of long loops, where ${\Nmin}\approx N^{\frac 2d}$ is the threshold between short and long. 
Under the PPP, these two numbers are independent, and we can write, in terms of a convolution,
\begin{equation}\label{convolution}
\PPP_{\L}\left(\Partn_{\L}=N-r\right)=\sum_k \PPP_{\L}\left({\rm N}^{\ssup{\rm short}}_{\L_N}=k\right) \PPP_{\L}\left({\rm N}^{\ssup{\rm long}}_{\L_N}=N-r-k\right)\, .
\end{equation}
The random variable $\frac1{|\L_N|}{\rm N}^{\ssup{\rm short}}_{\L_N}$ is strongly concentrated on the value $\rhoc$ (Proposition \ref{prop:ShortLoops}). Furthermore, for ${\rm N}^{\ssup{\rm long}}_{\L_N}$, Proposition \ref{prop:MedLongLoops} gives, for sufficiently large $N-r-k$, that
\begin{equation}
   \PPP_{\L}\left({\rm N}^{\ssup{\rm long}}_{\L_N}=N-r-k\right)
\sim \frac{\ex^{\gamma}}{{\Nmin}}\ex^{-\lambda_1\beta (N-k-r) \al^{-\frac 2d }}, 
\end{equation}
for some constant $\gamma\in(0,\infty)$. Using this once more for $N-k$ in place of $N-k-r$, we see that
\begin{equation}
   \PPP_{\L}\left({\rm N}^{\ssup{\rm long}}_{\L_N}=N-r-k\right)
\sim \PPP_{\L}\left({\rm N}^{\ssup{\rm long}}_{\L_N}=N-k\right)\ex^{\lambda_1\beta r \al^{-\frac 2d }}. 
\end{equation}
Using this in \eqref{convolution} shows that we obtain, again by convolution, the term $\PPP_{\L}\left(\Partn_{\L}=N-r\right)$ (i.e., the denominator in \eqref{eq:rdm_bc}), times an exponential that precisely cancels the exponential term in \eqref{gexpansion}. For these $r$, the summands in \eqref{eq:rdm_bc} have turned out to be asymptotically not dependent on $r$. The number of the $r$'s that contribute to this procedure is $\sim (\rho-\rhoc)|\L_N|$, which finishes the summary of the proof, as the volume term cancels out with that of~\eqref{gexpansion}.

\subsection{Preparation: eigenvalue expansion}\label{sec-EVexp}

We now express and approximate the fundamental solution $g_{\beta k}^{\ssup{\L}}$ and its trace $\q_k^{\ssup{\L}}=\int_{\L} g_{\beta k}^{\ssup{\L}}(x,x)\,\d x$ by using an eigenvalue expansion with respect to all the eigenvalues and eigenfunctions of the Laplacian $\frac12 \Delta^{\ssup{\L}}$ in the box $\L=L U$ with $U=[-\frac 12,\frac 12]^d$. For large $k$, we give upper and lower bounds that depend on spectral properties and will finally give the main terms in Proposition~\ref{prop:ODLRO_sup}, while for small $k$, we give upper bounds that will suffice to give the error terms in Proposition~\ref{prop:ODLRO_sup}.

We collect the relevant assertions in the following lemma. By $\lambda_2=\lambda_2^{\ssup{\rm bc}}> \lambda_1$ we denote the second-smallest eigenvalue of $-\frac12 \Delta^{\ssup{U,\rm bc}}$ in $U$.
\begin{lemma}[Asymptotics and bounds for $g_{k\beta}$ and $\q_k$]
\label{CorollaryQkSupercritical}
For any centred box $\L =L U$ and any periodic or diffusive boundary condition (which we suppress from the notation), the following hold:
\begin{enumerate}[label={(\roman*)}]
\item {\em ((Spectral) Asymptotics of $g_{k\beta}^{\ssup{\L}}$ for large $k$.)} Fix $c\in(0,\beta(\l_2-\l_1))$, then, uniformly in $x,y\in \L$ and in $k\geq 1$, as $L\to\infty$,
\begin{equation}\label{Equation6202}
\begin{split}
    \ex^{-\l_1 k\beta L^{-2}}\frac {\phi_1\rk{\smfrac xL}\phi_1\rk{\smfrac yL}}{L^d} &\leq  g_{k\beta}^{\ssup{\L}}(x,y) \\
    &\leq \ex^{-\l_1 k\beta L^{-2}}\Big(\frac {\phi_1\rk{\smfrac xL}\phi_1\rk{\smfrac yL}}{L^d}+\Ocal(k^{-d/2}\ex^{-ckL^{-2}})\Big)\, .
\end{split}
\end{equation}
\item {\em (Uniform upper bound for $g_{k\beta}^{\ssup{\L}}$.)}
For any $M>0$, there exist $c,c'>0$ such that, for all $L$,
\begin{equation}
    g_{k\beta}^{\ssup{\L}}(x,y)\le c g_{k\beta/c'}^{\ssup{\free}}(x,y),\qquad x,y\in\L, k\leq M L^2.
\end{equation}
\item {\em (Asymptotics of $\q_k^{\ssup{\L}}$ for large $k$.)} For any $\e>0$ and all $k$ satisfying $kL^{-2}\geq \e $,
\begin{equation}\label{eq:QkSupercritical}
   \ex^{-\lambda_1 \beta k L^{-2}} \leq  \q_k^{\ssup{\L}} \leq 
    \ex^{-\lambda_1 \beta k L^{-2}} \Big(1+\Ocal\big(\,\ex^{-\beta (\l_2-\l_1)k L^{-2}}\big)\Big).
\end{equation}
\item {\em (Asymptotics of $\q_k^{\ssup{\L}}$ for small $k$.)}
For any $k L^{-2}\le o(1)$, as $L\to\infty$,
\begin{equation}\label{eq:QkSubcritical}
    \Big|\frac{\rk{2\pi \beta k}^{d/2} \q_k^{\ssup{\L}}}{\al} -1\Big| \leq  1+\Ocal\big(\sqrt{ k L^{-2}}\big).
\end{equation}
\end{enumerate}
\end{lemma}

\begin{proof} The scaling properties of the Laplacian make it easy to go from the unit box $U=[-\frac 12,\frac 12]^d$ to the centred box $\L=L U = [-\frac 12 L,\frac 12 L]^d$:
\begin{equation}\label{scaling}
 g_{\beta k}^{\ssup{\L}}(x,y) =  L^{-d}   g_{\beta k/L^2 }(x/L,y/L),\qquad x,y\in \L,  k\in\N,
\end{equation}
where we wrote $g$ instead of $g^{\ssup{U}}$. In particular, writing $\q$ instead of $\q^{\ssup{U}}$,
\begin{equation}\label{scalingq}
\q_k^{\ssup{\L}}=\q_{kL^{-2}},\qquad k\in\N\, .
\end{equation}
Hence, it will be enough to prove the assertions for $g^{\ssup{\L}}$ replaced by $g$, $x/L$ and $y/L$ replaced by $x$ and $y$, and $k \beta L^{-2}$ replaced by $t\in(0,\infty)$.

Pick an orthonormal basis $(\phi_n)_{n\in\N}$ of $L^2(U)$ consisting of eigenfunctions corresponding to the sequence $(\lambda_n)_n$ of eigenvalues (ordered such that $\lambda_1<\lambda_2\leq \lambda_3\leq \dots$) then we have the spectral (Sturm--Liouville) decomposition 
    \begin{equation}\label{eq:spectral}
        g_t(x,y) = \sum_{n\in\N} \ex^{-t\lambda_n}\phi_n(x)\phi_n(y),\qquad t\in(0,\infty).
    \end{equation}
    Note that $\phi_1(x)>0$ for $x\in U\setminus \partial U$.

\begin{enumerate}[label={(\roman*)},itemindent=*,leftmargin=0pt]
\item The lower bound already follow from restricting the sum in \eqref{eq:spectral} to the first summand. Let us turn to the proof of the upper bound. For any $x\in U$ and $t\in(0,\infty)$, 
\begin{equation}
\abs{g_t(x,x) - \ex^{-t\l_1}\phi_1^2(x)} \leq \ex^{-\frac t2\l_2}\sum_{n\geq 2}\ex^{-\frac t2\l_n}\phi_n^2(x)
    = \ex^{-\frac t2\l_2} \big(g_{t/2}(x,x)-\ex^{-\frac t2\l_1}\phi_1^2(x)\big)\, .
\end{equation}
Iterating this $m\in\N$ times yields  
\begin{equation}
\abs{g_t(x,x) - \ex^{-t\l_1}\phi_1^2(x)}\leq C_{t 2^{-m}} \ex^{-t \l_2(m)},\qquad x\in  U,
\end{equation}
with $C_{s}=\sup_{x\in U}g_{s}(x,x)>0$ and using the short-hand notation $\l_2(1-2^{-m})=:\l_2(m)$.
From \eqref{gestifree} we obtain that $C_s=\Ocal\rk{s^{-d/2}}$. Furthermore, we have, for all $x,y\in U$,
\begin{equation*}
\begin{split}
   & \abs{g_t(x,y) - \ex^{-t\l_1}\phi_1(x)\phi_1(y) }\leq \ex^{-\frac t2\l_2}\sum_{n\geq 2}\ex^{-\frac t2\l_n}\abs{\phi_n(x)\phi_n(y)}\\
    &\quad\leq \ex^{-\frac t2\l_2}\sqrt{\sum_{n\geq 2}\ex^{-\frac t2 \l_n}\phi_n^2(x)} \sqrt{\sum_{n\geq 2}\ex^{-\frac t2 \l_n}\phi_n^2(y)}\\
    &\quad= \ex^{-\frac t2\l_2}\sqrt{g_{t/2}(x,x)-\ex^{-\frac t2\l_1}\phi_1^2(x)}\sqrt{g_{t/2}(y,y)-\ex^{-\frac t2\l_1}\phi_1^2(y)}\leq C_{t 2^{-m}} \ex^{- t\l_2(m)}\, . 
\end{split}
\end{equation*}
Picking  $m$ such that $\l_2(1-2^{-m})-\l_1>c$, this implies Assertion (i).

\item  According to \cite[Theorem 6.3.8]{a1981operator}, there are, for any $M>0$, constants $c,c',c''>0$ such that
\begin{equation}\label{gestifree}
    \abs{ g_{t}(x,y)-g_{t}^{\ssup{\free}}(x,y)}\le cg_{t/c'}^{\ssup{\free}}(x,y)\ex^{-c''\dist(y,\partial U)^2/t},\qquad x,y\in U, t\in(0,M].
\end{equation}
This implies (ii).

\item Putting $x=y$ and $t=k\beta$ in \eqref{eq:spectral} and integrating over $x\in U$, we obtain 
\begin{equation}\label{qkU}
    \q_k=\int_{U} g_{t}(x,x)\,\d x =\sum_{n\in\N}\ex^{-t\, \lambda_n } = \ex^{-t \l_1}\sum_{n\in\N}\ex^{-t\, (\lambda_n-\l_1) }\le \ex^{-t\l_1}\Big(1+\sum_{n=2}^\infty \ex^{-t\, (\lambda_n-\l_1) }\Big) .
\end{equation}
This implies Assertion (iii), as for $t\in [\e,\infty)$,
\begin{equation}
\begin{aligned}
\sum_{n=2}^\infty \ex^{-t\, (\lambda_n-\l_1) }
&= \ex^{-t (\l_2-\l_1)}\Big(1+ \sum_{n=3}^\infty\ex^{-t\, (\lambda_n-\l_2) }\Big)
\leq \ex^{-t (\l_2-\l_1)}\Big(1+ \ex^{\e \l_1} \sum_{n=1}^\infty\ex^{-\e\, \lambda_n}\Big)\\
&=\Ocal\rk{\ex^{-t (\l_2-\l_1)}},
\end{aligned}
\end{equation}
using \eqref{qkU} for $\e$ instead of $t$.

\item 
First note that it is enough to show \eqref{eq:QkSubcritical} for Dirichlet, Neumann, and periodic boundary conditions, as the following monotonicity holds, see~\cite[page 373]{a1981operator}: for $\bcd_1\le\bcd_2$,
\begin{equation}
    \q_k\hk{\bcd_1}\ge \q_k\hk{\bcd_2}\, .
\end{equation}
Hence, for any boundary condition $\bcd$,
\begin{equation}
    \q_k\hk{\dir}\le \q_k\hk{\bcd}\le \q_k\hk{\neu}\, .
\end{equation}
Hence, for the diffusive boundary conditions, it is enough to show the lower bound for Dirichlet condition and the upper bound for Neumann condition.

In $d=1$, the eigenvalues for the Laplacian $-\frac12 \Delta^{\ssup{U,\bc}}$ are given by
\begin{equation}
\lambda_n\Ssup{\dir} = \frac{\pi^2}2 n^2,\quad \lambda_n\Ssup{\neu} = \frac{\pi^2}2 (n-1)^2,\quad \lambda_1\Ssup{\per} =0,\ \lambda_{2n}\Ssup{\per} = \lambda_{2n+1}\Ssup{\per}= 2\pi^2 n^2,\quad\mbox{for }n\in\N.
\end{equation}

The $d$-dimensional eigenvalues are then of the form $\frac{\pi^2}2\sum_{i=1}^d n_i^2$ for Dirichlet b.c., $\frac{\pi^2}2\sum_{i=1}^d (n_i-1)^2$ for Neumann b.c., and $2\pi^2\sum_{i=1}^d (n_i-1)^2$ for periodic b.c., with $n_1,\dots,n_d\in\N$. They are simple for Dirichlet and Neumann b.c., while they have multiplicity two in the periodic case (with the exception of $\l_1=0$).

We first show the lower bound for Dirichlet boundary conditions. 
In the limit as $t\downarrow0$,
\begin{equation}
\begin{aligned}
\q_t^{\ssup{U, \dir}}&= \sum_{n\in\N}\ex^{-t\lambda_n^{\ssup{U, \dir}}}\geq\Big(\sum_{m\in\N}\ex^{-t \frac {\pi^2}2 m^2}\Big)^d\geq \Big(\int_1^\infty {\rm e}^{-t\pi^2 \frac{x^2}2}\, {\rm d} x\Big)^d\\
&\geq \Big(\int_0^\infty {\rm e}^{-t\pi^2 \frac{x^2}2}\, {\rm d} x-1\Big)^d
=\big((2t \pi)^{-1/2}-1\big)^d=(2t\pi)^{-d/2}\big(1-\Ocal(\sqrt t)\big)\, ,
\end{aligned}
\end{equation}
which implies the lower bound in \eqref{eq:QkSubcritical}. For periodic boundary conditions, we have, since the non-zero eigenvalues have multiplicity $2$, 
\begin{equation}
\begin{aligned}
\q_t^{\ssup{U, \per}} 
&= \sum_{n\in\N}\ex^{-t \lambda_n^{\ssup{U, \per}}}= \Big(1+2\sum_{m\in\N}\ex^{-t 2\pi^2 m^2}\Big)^d
\geq \Big(1+2 \int_1^\infty\ex^{-4t\pi^2\frac{x^2}2}\,\d x\Big)^d\\
&=\Big(1+\frac 1{\sqrt{2t\pi}}-2\int_0^1 \ex^{-4t\pi^2\frac{x^2}2}\,\d x\Big)^d \geq \frac 1{(2t\pi)^{d/2}}\Big(1-\Ocal\big(\sqrt{ t}\big)\Big),
\end{aligned}
\end{equation}
where the last step was in the limit $t\downarrow0$.

We now show the upper bound for Neumann b.c.. 
For any $t\in(0,\infty)$,
\begin{equation}\label{eq:NP:q_kLargek}
\begin{aligned}
    \q_t^{\ssup{U, \neu}}&= \Big(\sum_{m\in\N}\ex^{-t \frac{\pi^2}2(m-1)^2}\Big)^d= \Big(1+\sum_{m\in\N}\ex^{-t \frac{\pi^2}2 m^2}\Big)^d\\
     &\leq \Big(1+\int_0^\infty \ex^{-t\pi^2\frac {x^2}2}\,\d x\Big)^d=\Big(1+\frac 1{\sqrt{2t\pi}}\Big)^d
    =\frac { 1+\Ocal(\sqrt t)}{(2t\pi)^{d/2}},
     \end{aligned}
\end{equation}
where the last assertion is in the limit $t\downarrow 0$.

For periodic boundary conditions, we see that
\begin{equation}
\begin{split}
    \q_t^{\ssup{U, \per}}&= \Big(1+2\sum_{m\in\N}\ex^{-t 2\pi^2 m^2}\Big)^d<\Big(1 + 2\int_0^\infty\ex^{4t\pi^2\frac{x^2}2}\,\d x\Big)^d\\
    &= \Big(1+\frac 1{\sqrt{2t\pi}}\Big)^d=\frac { 1+\Ocal(\sqrt t)}{(2t\pi)^{d/2}}\, ,
\end{split}
\end{equation}
with the same upper bound as above. 
\end{enumerate}
\end{proof}
In the following, we will often use the estimates of the above lemma with some generic constant $C\in(0,\infty)$ that may change its value from appearance to appearance; it may depend only on $d$, $\beta$, $\rho$ and $\bc$.

\subsection{Particle numbers in small loops in the Poisson point process}\label{sec-smalloops}

We consider now the Poisson point process $(\omega_i)_i$ on $\N$ that we introduced in Section~\ref{sec-PPP}. Recall that we call each Poisson point $\omega$ at $k\in\N $ {\em a loop of length} $k$, and it contains $k$ particles. Then $X_k$ denotes their cardinality, and $X_k$ is Poisson-distributed with parameter $\frac 1k \q_k^{\ssup{\L_N,{\rm Dir}}}$, and $(X_k)_{k\in\N}$ is independent. Also recall that a loop $\omega$ is said to be \emph{short} if its length is less or equal to $\Nmin \defeq \lceil N^{\frac 2d }\log^{1/2}(N)\rceil$, and $\Ns_\L$ denotes the total particle number in all the short loops in $\L\subset\R^d$. Otherwise, $\omega$ is called \emph{long}, and $\Nl_\L$ denotes the number of particles in long loops in $\L$.

In this section we study the number of particles in short loops and the entire particle number in terms of a law of large numbers with exponential bounds for the probability of a deviation from the asymptotic mean. We split into loops of bounded lengths and longer ones, and finally we consider the total number of particles. In this section we use standard probabilistic tools for sums of independent Poisson random variables with various parameters, mainly Markov's inequality for exponential functions.

Let us first turn to the loops of bounded length. With a parameter $R\in\N$, we recall that ${\rm N}_\L^{\sqsup {1,R}}$ is the number of particles in loops of lengths $\leq R$ in $\L$ (we call these loops $R$-short). In the next proposition, we show that ${\rm N}_\L^{\sqsup {1,R}}$ concentrates with stretched-exponentially small probability on $\rhoc^{\ssup R}|\L|$, where $\rhoc^{\ssup R}=(2\pi\beta)^{-d/2}\sum_{k=1}^R k^{-d/2}$, which converges to $\rhoc$ for $R\to\infty$. As in Proposition~\ref{prop:ODLRO_sup}, we fix a diffusive or periodic boundary condition and suppress it from the notation.

\begin{proposition}[Number of particles in $R$-short loops]\label{prop:ShortLoops}
Fix $\rho\in(\rhoc,\infty)$ and pick the centred box $\L_N$ with diameter $L_N$ such that $L_N^d=N/\rho$. Fix $R\in\N$, then
\begin{equation}\label{eq:MeanShort}
    \E^{\ssup{{\rm bc},N}}_{\L_N}[{\rm N}^{\sqsup {1,R}}_{\L_N}] \sim \rhoc^{\ssup{R}}\alN,\qquad N\to\infty.
\end{equation} 
Furthermore, for any $\e>0$, there exists $C_{\e}>0$ such that, for all $N\in\N$,
\begin{equation}\label{eq:UpperBoundShortLoops}
    \P^{\ssup{{\rm bc},N}}_{\L_N}\rk{\abs{\frac1{|\L_N |}{\rm N}^{\sqsup {1,R}}_{\L_N}-\rhoc^{\ssup R}}>\e}\le \ex^{-C_\e \alN}\, .
\end{equation}
\end{proposition}

\begin{proof} 
We write $\L$ for $\L_N$. We use the Campbell formula $\E^{\ssup{{\rm bc},N}}_\L[{\rm N}^{\sqsup {1,R}}_\L] =\sum_{j=1}^R\q^{\ssup{\L,\bc}}_j$ and note that $\q^{\ssup{\L,\bc}}_j  \sim \al(2\pi \beta j)^{-d/2}$ as $N\to\infty$ by \eqref{eq:QkSubcritical}, which implies \eqref{eq:MeanShort}.

Now we prove \eqref{eq:UpperBoundShortLoops}. According to \eqref{eq:MeanShort}, we are considering the probability of deviations of ${\rm N}^{\sqsup {1,R}}_\L$ from its mean. 
We handle only the upper deviations, since the lower deviations are handled in the same way (with negative $s$ instead of positive $s$). By the exponential Chebyshev inequality, we have that for any $s>0$,
\begin{equation}\label{EquationMarkovInequalty}
\begin{aligned}
    \P^{\ssup{{\rm bc},N}}_\L\rk{{{\rm N}^{\sqsup {1,R}}_\L-\rhoc^{\ssup R}\al}>\e\al}
    &\le \ex^{-s\e \al}\E^{\ssup{{\rm bc},N}}_\L\ek{\ex^{s\rk{{\rm N}^{\sqsup {1,R}}_\L-\E^{\ssup{{\rm bc},N}}_\L\ek{{\rm N}^{\ssup{\leq R}}_\L}}}}\\
    &=\ex^{-s\e \al}\exp\rk{\sum_{j=1}^{R}\frac{1}{j}\rk{\ex^{s j}-1-sj}\q^{\ssup{\L,\bc}}_j}\, ,
\end{aligned}
\end{equation}
where we used Campbell's formula in the second step. Note that $\frac{1}{j}\rk{\ex^{s j}-1-sj}=\Ocal\rk{s^2}$ for $s\downarrow 0$. Also using that $\sum_{j=1}^{R} \q^{\ssup{\L,\bc}}_j\sim \rhoc^{\ssup R}|\L|$, we see that
\begin{equation}
    \sum_{j=1}^{R}\frac{1}{j}\rk{\ex^{s j}-1-sj}\q^{\ssup{\L,\bc}}_j=\Ocal\rk{s^2\al}\, .
\end{equation}
By choosing $s$ sufficiently small (not depending on $\L$, but on $\e$), the linear term in the exponential in \eqref{EquationMarkovInequalty} dominates. This implies \eqref{eq:UpperBoundShortLoops}.
\end{proof}

Now we control the deviations of the number ${\rm N}_\L^{\sqsup{R+1,{\Nmin}}}=\Ns_\L-{\rm N}_\L^{\sqsup {1,R}} $ of particles in loops of lengths between $R+1$ and ${\Nmin}$. It turns out that ${\rm N}_\L^{\sqsup{R+1,{\Nmin}}}$ is strongly concentrated on $(\rhoc-\rhoc^{\ssup R})|\L|$. Even more, the probability of the deviations again decays exponentially, however, not in $|\L_N|$, but in $|\L_N|^{1-d/2}$, but with an arbitrarily large prefactor. The proof is similar to the one of Proposition~\ref{prop:ShortLoops}.

\begin{proposition}[Particles in short, but not $R$-short loops]\label{prop:MeanShort}Fix $\rho\in(\rhoc,\infty)$ and let $\L_N$ be the centred box with volume $N/\rho$.    Fix $R\in\N$, then
\begin{equation}\label{eq:MeanLShort}
    \E^{\ssup{{\rm bc},N}}_{\L_N}[{\rm N}_{\L_N}^{\sqsup{R+1,{\Nmin}}}] \sim (\rhoc-\rhoc^{\ssup{R}})\alN,\qquad N\to\infty.
\end{equation}
Furthermore, for any $\e>0$ and for any $\kappa>0$, we have, for all large $N\in\N$,
\begin{equation}\label{eq:UpperBoundLShortLoops}
    \PPP^{\ssup{{\rm bc},N}}_{\L_N}\rk{\abs{\frac 1{|\L_N|}{\rm N}_{\L_N}^{\sqsup{R+1,{\Nmin}}}-(\rhoc-\rhoc^{\ssup R})}>\e}\le \ex^{-\kappa \alN^{1-d/2}}\, .
\end{equation}
Furthermore, the latter assertion remains true if $\e$ is replaced by a sequence $\e_N$, tending to zero as $N\to\infty$ slowly (say at logarithmic speed in $N$).
\end{proposition}

\begin{proof} Again, we write $\L$ instead of $\L_N$, and we drop the superscript \lq$N$\rq\ in the probability measure. We use Campbell's formula and split, for small $\delta>0$,
\begin{equation}\label{eq:expSubcrit}
    \Ebc_\L[{\rm N}_\L^{\sqsup{R+1,{\Nmin}}}] =\sum_{k=R+1}^{N^{\frac 2d-\delta}}\q_k^\ssup{\L,\bc}+\sum_{N^{\frac 2d-\delta}<k\leq \Nmin}\q_k^\ssup{\L,\bc}\, .
\end{equation}
By \eqref{eq:QkSubcritical}, we get that
\begin{equation}\label{eq:sumSubcrit}
  \sum_{k=R+1}^{N^{\frac 2d-\delta}}\q_k^\ssup{\L,\bc}  \sim \al\sum_{k=R+1}^{N^{\frac 2d-\delta}}\frac{1}{\rk{2\pi\beta j}^{d/2}}\sim (\rhoc-\rhoc^{\ssup R})\al \, ,
\end{equation}
and on the other hand,
\begin{equation}\label{sumonlargek}
    \sum_{N^{\frac 2d-\delta}<k\leq \Nmin}\q_k^\ssup{\L,\bc}\le \Ocal\rk{\al}\sum_{N^{\frac 2d-\delta}<k\leq \Nmin}\frac{1}{\rk{2\pi\beta k}^{d/2}}=\Ocal(|\L|^{\frac 2d+\delta(\frac d2-1)})=o\rk{\al}\, .
\end{equation}
This concludes the proof of \eqref{eq:MeanLShort}.

Now we prove \eqref{eq:UpperBoundLShortLoops}. First we proceed with estimating the probability of deviation from below. We use the exponential Chebyshev inequality with some $\alpha\in(0,\infty)$, to obtain
      \begin{equation}
          \PPP^{\ssup{{\rm bc}}}_{\L}\rk{\frac 1{|\L|}{{\rm N}^\sqsup{R+1,{\Nmin}}_{\L}\leq \rhoc-\rhoc^{\ssup R}-\e}}\leq \ex^{\alpha\al(\rhoc-\rhoc^{\ssup R}-\e)}\EPP^{\ssup{{\rm bc}}}_{\L}\Big[\ex^{-\alpha{\rm N}^\sqsup{R+1,{\Nmin}}_{\L}}\Big]\, .
          \end{equation} 
Then we use the Campbell formula and the estimate, for some small $\delta>0$, according to Lemma~\ref{CorollaryQkSupercritical}(iv), $\q_k=(1+o(1)) |\L| (2\pi \beta k)^{-d/2} $ for $k\leq L^{2-\delta}$ to get, picking 
      $\alpha =sL^{-2}$ for some $s\in(0,\infty)$,
\begin{equation}
\begin{aligned}
    \EPP^{\ssup{{\rm bc}}}_{\L}&\Big[\ex^{-\alpha{\rm N}^\sqsup{R+1,{\Nmin}}_{\L}}\Big] =\exp\Big\{\sum_{k=R+1}^{\Nmin}(\ex^{-\alpha k}-1)\frac  {\q_k} k\Big\}\\
    &\leq \exp\Big\{ (1+o(1))\sum_{k=R+1}^{L^{2-\delta}}\rk{\ex^{-skL^{-2}}-1}\frac{\al}{(2\pi\beta k)^{d/2}}\frac 1 k\Big\}\\
          &\leq \exp\Big\{ -(1+o(1))\frac{L^d}{(2\pi\beta)^{d/2}} sL^{-2}\zeta(d/2)\Big\}
          = \exp\Big\{ -s\rhoc L^{d-2}\rk{1+o(1)}\Big\}\,,
\end{aligned}
\end{equation}
where we estimated $\ex^{-w}-1\leq -(1+o(1)) w$ for $w\in[0,s L^{-\delta}]$ and $\sum_{k=R+1}^{L^{2-\delta}} k^{-d/2}\geq \zeta(d/2)(1+o(1))$. Taking $s= \kappa/\e$ yields the lower-deviations claim.

To estimate the probability of the upper deviations, we proceed at the beginning analogously with \lq$\leq\rhoc-\rhoc^{\ssup R}-\e$\rq\ replaced by \lq$\geq\rhoc-\rhoc^{\ssup R}+\e$\rq, and have now to handle positive exponential moments:
\begin{equation}
\begin{aligned}
    \EPP^{\ssup{{\rm bc}}}_{\L}\Big[\ex^{\alpha{\rm N}^\sqsup{R+1,{\Nmin}}_{\L}}\Big]        &=\exp\Big\{\sum_{k=R+1}^{\Nmin}(\ex^{\alpha k}-1)\frac  {\q_k} k\Big\}\, .
\end{aligned}
\end{equation}
The sum on $k=R +1,\dots,\Nmin$ is split into the some on $k=R+1,\dots, L^{2-\delta}$ and the remainder for some small $\delta>0$. For $k$ in the first sum, we use again  Lemma~\ref{CorollaryQkSupercritical}(iv) to write $\q_k=(1+o(1)) |\L| (2\pi \beta k)^{-d/2} $. For the exponential term, this time we use  $\ex^{w}-1\leq w+w^2 \ex^{w}$ for $w\in(0,s L^{-\delta}]$. Hence, the first sum gives
\begin{equation}
\begin{aligned}
    \sum_{k=R+1}^{L^{2-\delta}}&(\ex^{\alpha k}-1)\frac  {\q_k} k
    \\
    &=\rk{1+o(1)}\frac{ L^d}{(2\pi\beta)^{d/2}}\Big[ sL^{-2}\sum_{k=R+1}^{L^{2-\delta}}k^{-d/2}+s^2 L^{-4}\ex^{s L^{-\delta}}\sum_{k=R+1}^{L^{2-\delta}}k^{-d/2+1}\Big]\\
    &\leq \rk{1+o(1)} s L^{d-2}\Big[\rhoc-\rhoc^{\ssup R}+ L^{-2}\ex^{s L^{-\delta}}\sum_{k=1}^{L^{2-\delta}}k^{-d/2+1}\Big]\, .
\end{aligned}
\end{equation}
Now we see that the last term in the square brackets vanishes as $N\to\infty$, by distinguishing $d=3$, $d=4$ and $d\geq 5$. Hence, this part of the sum in the exponential term is not larger than $(1+o(1)) s L^{d-2}(\rhoc-\rhoc^{\ssup R})$.

Next, we turn to the sum on $k=L^{2-\delta},\dots,\Nmin$. We use \eqref{eq:QkSupercritical} to estimate crudely $\q_k^\ssup{\L,\bc}\leq C \ex^{-k \beta \lambda_1 L^{-2}}\leq C$ and obtain, putting $\alpha=s L^{-2}$, and recalling that $\Nmin=L^2 \log^{1/2}(N)$,
\begin{equation}
\sum_{k=L^{2-\delta}}^{\Nmin}(\ex^{\alpha k}-1)\frac  {\q_k} k
\leq C \ex^{s\log^{1/2}(N)}\sum_{k=L^{2-\delta}}^{\Nmin}\frac 1k
\leq C\ex^{s\log^{1/2}(N)} \log \big(L^\delta\log^{1/2}(N)\big)\, ,
\end{equation}
which is $o(L^{d-2})$. Summarising, we get
\begin{equation}
\begin{split}
    \PPP^{\ssup{{\rm bc}}}_{\L}&\rk{\frac 1{|\L|}{{\rm N}^\sqsup{R+1,{\Nmin}}_{\L}\geq \rhoc-\rhoc^{\ssup R}+\e}}\\
    &\leq \ex^{s L^{-2}\al(\rhoc-\rhoc^{\ssup R}-\e)}\ex^{(1+o(1)) s L^{d-2}(\rhoc-\rhoc^{\ssup R})} \ex^{-(1+o(1)) s L^{d-2}\e}.
\end{split}
\end{equation} 
Taking $s=\kappa/\eps$, the proof of \eqref{eq:UpperBoundLShortLoops} is finished.

The additional assertion is seen to follow from the above proof by taking $s=s_N= \kappa/\e_N$, tending slowly to infinity in a way that all the steps in the above proof remain true.
\end{proof}

Putting together Propositions~\ref{prop:ShortLoops} and \ref{prop:MeanShort}, we obtain the corresponding estimate for $\Ns_\L$.

\begin{cor}[Deviations of the particle number in short loops]\label{Cor:shortloopsdev}
For any $\e>0$ and for any $\kappa>0$, we have,  for all large  $N\in\N$,
\begin{equation}\label{eq:UpperBoundLShortLoopsjoint}
    \PPP^{\ssup{{\rm bc},N}}_{\L_N}\rk{\abs{\frac 1{|\L_N|}{\rm N}_{\L_N}^{\ssup{\rm short}}-\rhoc}>\e}\le \ex^{-\kappa \alN^{1-d/2}}\, .
\end{equation}
The latter assertion remains true if $\e$ is replaced by a sequence $\e_N$, tending to zero as $N\to\infty$ suitably slowly.
\end{cor}

Let us now give a rough lower bound for the probability that the total number of particles is equal to $N$, i.e., for the term in the denominator in \eqref{eq:rdm_bc}. We fix a diffusive or periodic boundary conditions.

\begin{lemma}[Lower bound for the denominator in~\eqref{eq:rdm_bc}]\label{PropsotionUpperBoundTotal}
Fix $\rho\in(\rhoc,\infty)$ and let $\L_N$ be the centred box with volume $N/\rho$. Then, if $\l_1>0$,
\begin{equation}\label{eq:UpperBoundTotal}
    \PPP^{\ssup{{\bc},N}}_{\L_N}\rk{\Partn_{\L_N}=N}\geq \ex^{-(\rho-\rhoc)\beta\lambda_1\alN^{1-\frac 2d }\rk{1+o(1)}}\,,\qquad N\to\infty .
\end{equation}
If $\l_1=0$, the lower bound is polynomial: for any $\eps>0$,
\begin{equation}\label{eq:UpperBoundTotal_0}
     \PPP^{\ssup{{\bc},N}}_{\L_N}\rk{\Partn_{\L_N}=N}\geq |\L_N|^{-2-\frac 2d-\eps},\qquad N\to\infty.
\end{equation}
\end{lemma}
\begin{proof} We abbreviate $\L=\L_N$ and drop the super-indices. Recall that $\Partn_{\L}=\sum_{k=1}^N k X_k$, where $X_1,\dots,X_N$ are independent Poisson-distributed variables with parameters $\q_1,\dots,\q_N$.
We pick $\delta\in(0,\frac 2d(1-\frac 2d))$, $R=R_N=N^{\frac 2d -\delta}$, a sequence $\eps_N\downarrow 0$, and estimate
\begin{equation}
\begin{split}
&\PPP_{\L_N}\rk{\Partn_{\L_N}=N}=\P\Big(\sum_{j=1}^N j X_j=N\Big)\\
&\geq 
\sum_{m\in [(\rhoc-\eps_N)|\L|,(\rhoc+\eps_N)|\L|]}\, \P\Big(\sum_{j\leq R} j X_j=m\Big)\\
&\phantom{\geq 
\sum_{m\in [(\rhoc-\eps_N)|\L|,(\rhoc+\eps_N)|\L|]}\quad}
\times \P(X_{N-m}=1)\, \P(X_j=0\ \forall j\notin [1,R]\cup\{N-m\})\, .
\end{split}
\end{equation}
Then, uniformly for any $m$ in the above sum, we have $N-m\sim (\rho-\rhoc)|\L|$. First note that $\E[\sum_{j\leq R}j X_j]\sim \rhoc|\L|$ (see Proposition~\ref{prop:MeanShort}) and the standard deviation of $\sum_{j\leq R}j X_j$ is $o(|\L|)$, hence, $\sum_{m\in [(\rhoc-\eps_N)|\L|,(\rhoc+\eps_N)|\L|]} \P(\sum_{j\leq R} j X_j=m)$ tends to one as $N\to\infty$, if $\eps_N$ vanishes at a sufficiently low speed. Moreover:

If $\l_1>0$, we use Lemma~\ref{CorollaryQkSupercritical}(iii) to obtain
\begin{equation*}
\P(X_{N-m}=1)=\ex^{-\frac 1{N-m} \q^{\ssup{\L}}_{N-m}}\frac { \q^{\ssup{\L}}_{N-m}}{N-m} \sim \frac {\q^{\ssup{\L}}_{N-m}}{N-m} = \ex^{-\beta \lambda_1 (\rho-\rhoc)|\L|^{1-\frac 2d}(1+o(1))},\qquad N\to\infty\, .
\end{equation*}
The zero probability for the remaining Poisson variables is estimated as follows, with the help of Lemma~\ref{CorollaryQkSupercritical}(iii) and (iv).
\begin{equation}
\begin{aligned}
    \P(X_j=0\ \forall j\notin [1,R_N]\cup\{N-m\}) &\geq\ex^{-\sum_{j=R_N}^N\frac{1}{j}\q_j}
    \geq \ex^{-\Ocal(L^d) \sum_{j=R_N}^{L^2}j^{-1-d/2}+o(1)}\\
    &=\ex^{-\Ocal(|\L|^{\delta\frac d 2 }) } \geq \ex^{-o(|\L|^{1-\frac 2d})}\, ,
\end{aligned}
\end{equation}
the last step follows from our assumption that $\delta<\frac 2d(1-\frac 2d)$. This finishes the proof of \eqref{eq:UpperBoundTotal}.

In the case $\l_1=0$, we pick a small $\eps>0$ and adapt $\delta$ later. We use Lemma~\ref{CorollaryQkSupercritical}(iii) to estimate 
\begin{equation}
    \P(X_{N-m}=1)=\ex^{-\frac 1{N-m} \q^{\ssup{\L}}_{N-m}}\frac { \q^{\ssup{\L}}_{N-m}}{N-m}\sim\frac { \q^{\ssup{\L}}_{N-m}}{N-m}\geq \frac 1{N-m}\geq \Ocal(\smfrac 1N).
\end{equation}
In order to estimate the zero probability of the other Poisson variables, we estimate, using $\q_j\leq 1+\Ocal(\ex^{-\beta\l_2 jL^{-2}})$, 
\begin{equation}
    \ex^{-\sum_{j=R_N}^N\frac{1}{j}\q_j}\geq \exp\Big(-\log\frac{N}{N^{\frac2d-\delta}} - C\sum_{j=R_N}^N \frac{1}{L^{-2}j}\ex^{-\beta\l_2 jL^{-2}}L^{-2}\Big)\, .
\end{equation}
With an integral comparison, the second term in the exponential can be upper bounded by $\int_1^\infty \d x\,\frac{1}{x}\ex^{-\beta\l_2 x} + \sum_{j=R_N}^{L^2}\frac1j\leq C+ \delta d\log L$, so we have
\begin{equation}
    \ex^{-\sum_{j=R_N}^N\frac{1}{j}\q_j}\geq N^{-\rk{1-\frac2d+\delta}}C\ex^{-C\delta d\log L}\,,
\end{equation}
which is not smaller than $|\L_N|^{-1-\frac 2d-\eps}$, if $\delta$ is picked small enough.
\end{proof}

\subsection{Particle numbers in long loops in the Poisson point process}\label{sec-longloops}

We are now interested in the distribution of the number $\Nl_\L$ of particles in long loops, i.e., longer than ${\Nmin}=\lceil|\L|^ {\frac 2d}\log^{1/2} N\rceil$. For these loops, we need to find precise asymptotics (i.e., up to a factor of $1+o(1)$) for the probability that $\Nl_\L$ is equal to a large number. As this is rather subtle, we need to use finer and non-standard means. In Proposition~\ref{prop:MedLongLoops}, below, we prove a slightly more general assertion, which will then be used in the identification of the limiting distribution of the long loops. 

We will be able to use some well-known fine asymptotics and limiting assertions about uniformly-distributed random partitions from~\cite{arratia2003logarithmic}. For this, we need the density $p\colon (0,\infty)\to[0,\infty)$ of the distribution on $(0,\infty)$ with Laplace-transform
\begin{equation}\label{LaplacetrafoPD}
    (0,\infty)\ni s\mapsto \exp\rk{-\int_0^1\rk{1-\ex^{-sx}\frac{1}{x}}\,\d x}.
\end{equation}
For $x\in[0,1]$, the value of $p(x)$ is constant and is given by $\ex^{-\gamma }$, 
where $\gamma\approx 0.5772$ is the Euler--Mascheroni constant.

Denote $\Nmax= \lceil N^{2/d}\log^2(N)\rceil$.

\begin{proposition}[Distribution of the number of particles in long loops]\label{prop:MedLongLoops}
Fix $\rho\in(\rhoc,\infty)$ and let $\L_N$ be the centred box with volume $N/\rho$. 

Then, for any $y\in(0,\infty)$, in the limit as $N\to\infty$, uniformly in $\Nmax\leq \M\leq k\leq N$ such that $k/\M\to y$,
\begin{equation}\label{eq:asylongloops}
    \PPP^{\ssup{\bc,N}}_{\L_N}\rk{{\rm N}^{\sqsup{\Nmin+1,\M}}_{\L_N}=k}\sim p(y)\ex^{-\beta\lambda_1 k\abs{\L_N}^{-\frac 2d }}\times
    \begin{cases}\frac  1\Nmin
    &\text{if }\l_1>0\, ,\vspace{2mm}\\
     \frac 1 \M&\text{if }\l_1=0\, .
 \end{cases}
\end{equation}
In particular, in the limit as $N\to\infty$, uniformly in $\Nmax\leq k\leq N$,
\begin{equation}\label{eq:LowerBoundSLongLOops}
    \PPP^{\ssup{\bc,N}}_{\L_N}\rk{\Nl_{\L_N}=k}\sim\ex^{-\gamma}
    \ex^{-\beta\lambda_1 k\abs{\L_N}^{-\frac 2d }}\times
    \begin{cases}\frac  1\Nmin
    &\text{if }\l_1>0\, ,\vspace{2mm}\\
    \frac 1{N}&\text{if }\l_1=0\, .
 \end{cases}
\end{equation}
Furthermore, there exists a $C\in(0,\infty)$ such that
\begin{equation}\label{eq:UpperBoundSLongLoops}
    \PPP^{\ssup {\bc,N}}_{\L_N}\rk{\Nl_{\L_N}=k}\leq  C^{k/\Nmin}\ex^{-\beta\lambda_1 k\abs{\L_N}^{-\frac 2d }}\, ,\qquad k\in\{1,\dots,N \}, N\in\N.
\end{equation}
\end{proposition}

\begin{proof} 
Recall from \eqref{partitions} the set $\mathfrak P_k$ of all partitions $\partition$ satisfying $\sum_i i \partition_i=k$ of $k$. Denote by
\begin{equation}\label{Pk*def}
    \Pfrak_k^{\sqsup{\Nmin+1,\M}}=\big\{\partition\in\Pfrak_k\colon \partition_i=0\text{ for }i\notin [\Nmin+1,\M] \big\}
\end{equation} 
the set of all partitions of $k$ with partition sets of sizes in $[\Nmin+1,\M]$. Again, we write $\L$ instead of $\L_N$, drop the boundary condition from the notation, and put $L^d=|\L|$. Assume that $\M,k\in\{N^+,N^+ +1,\dots,N\}$ with $\M\leq k$.

Recall that ${\rm N}^{\sqsup{\Nmin+1,\M}}_{\L}$ is equal in distribution to $\sum_{r=\Nmin+1}^{\M} r X_r$, where $X_1,\dots,X_N$ are independent Poisson-distributed random variables, and $X_r$ has parameter $\frac 1 r \q^{\ssup \L}_r$. Hence, the exact value of the probability in question is
\begin{equation}\label{longloopsdistribution}
    \PPP_\L\rk{{\rm N}^{\sqsup{\Nmin+1,\M}}_{\L}=k}
    =\ex^{-\sum_{r=\Nmin+1}^j\frac 1r \q^{\ssup \L}_r}\sum_{\partition\in\Pfrak_k^{\sqsup{\Nmin+1,\M}}}\prod_{r=\Nmin+1}^{\M}\frac{\rk{\q^{\ssup \L}_{r}}^{\partition_r}}{\partition_r!\,r^{\partition_r}}\, .
\end{equation}
We first compute the exponential prefactor in the first term of the right-hand side of the above equation. By Lemma~\ref{CorollaryQkSupercritical}(iv), we have, uniformly in $r\in\{ \Nmin+1,\dots,N\}$, the asymptotics $ \q^{\ssup \L}_r= \ex^{-\lambda_1 r\beta L^{-2}}(1+\ex^{-c \log^{1/2}(N)})$ for some $c\in(0,1)$ and hence 
\begin{equation}\label{Tsumasy}
\begin{aligned}
    \sum_{r= \Nmin+1}^\M\frac 1r\q^{\ssup \L}_r
    &=(1+\ex^{-c \log^{1/2}(N)}) L^{-2}\sum_{r= \Nmin+1}^\M\frac 1{rL^{-2}}\ex^{-\lambda_1 r\beta L^{-2}}\\
    &=(1+\ex^{-c \log^{1/2}(N)})\int_{{(\Nmin+1)} L^{-2}}^{\M L^{-2}} \frac 1x \ex^{-\lambda_1 x\beta}\, \d x \\
    &=
    \begin{cases}(1+\ex^{-c \log^{1/2}(N)})\log\frac \M{\Nmin}&\mbox{if }\l_1=0\, ,\\
    o(1)&\mbox{if }\l_1>0\, .
    \end{cases}
\end{aligned}
\end{equation}
Hence the first term on the right-hand side of~\eqref{longloopsdistribution} is $\frac \Nmin \M(1+o(1))$ if $\l_1=0$ and $1+o(1)$ if $\l_1>0$. Therefore, for proving~\eqref{eq:asylongloops}, it only remains to show that, uniformly in $\Nmax\leq \M \leq k\leq N$ such that $k/\M\to y\in(0,\infty)$,
\begin{equation}\label{Prop28(1)}
    \sum_{\partition\in\Pfrak_k^{\sqsup{\Nmin+1,\M}}}\prod_{r=\Nmin+1}^{\M}\frac{\rk{\q^{\ssup \L}_{r}}^{\partition_r}}{\partition_r!\,r^{\partition_r}}
\sim \frac{p_1(y)}{{\Nmin}}\ex^{-\beta\lambda_1 k\abs{\L_N}^{-\frac 2d }},\qquad N\to\infty\, .
\end{equation}

\textbf{Proof of lower bound in~\eqref{Prop28(1)}:} Using the lower bound in Lemma~\ref{CorollaryQkSupercritical}(iii), and using that $\sum_{r=\Nmin+1}^{\M} r \partition_r=k$ for $\partition\in \Pfrak_k^{\sqsup{\Nmin+1,\M}}$, we have
\begin{equation}\label{Equation048231}
    \sum_{\partition\in\Pfrak_k^{\sqsup{\Nmin+1,\M}}}\prod_{r=\Nmin+1}^{\M}\frac{\rk{\q^{\ssup \L}_{r}}^{\partition_r}}{\partition_r!\,r^{\partition_r}}
    \ge \ex^{-\lambda_1\beta k\al^{-2/d}}\sum_{\partition\in\Pfrak_k^{\sqsup{\Nmin+1,\M}}}\prod_{r=\Nmin+1}^{\M}\frac{1}{\partition_r!\,r^{\partition_r}}\, .
\end{equation}
Consider the vector $(Y_r)_{1\leq r\leq k}$ of independent Poisson-distributed random variables such that $Y_r$ has parameter $1/r$, then~\cite[Thm.~4.13]{arratia2003logarithmic} states that
\begin{equation}\label{LogBookThm4.13}
    \P\Big(\sum_{r=\Nmin+1}^{\M} r Y_r=k\Big)\sim\frac{p_1(y)}{\M}\, ,
\end{equation}
where we recall that we assumed that $k/\M\to y$. Hence, the sum on the right-hand side of~\eqref{Equation048231} has the asymptotics
\begin{equation}\label{Equation731233}
    \sum_{\partition\in\Pfrak_k^{\sqsup{\Nmin+1,\M}}}\prod_{r=\Nmin+1}^{\M}\frac{1}{\partition_r!\,r^{\partition_r}}
    =\ex^{\sum_{r=\Nmin+1}^{\M}\frac{1}{r}}\P\rk{\sum_{r=\Nmin+1}^{\M} rY_r=k}\sim \frac{\M}{\Nmin}\frac{p_1(y)}{\M}= \frac{p_1(y)}{\Nmin}\, ,
\end{equation}
according to the asymptotics $\sum_{r=1}^m\frac 1r\sim\log m$ as $m\to\infty$. This finishes the proof of the lower bound in \eqref{Prop28(1)} and hence the lower bound in \eqref{eq:asylongloops}.

\textbf{Proof of~\eqref{eq:UpperBoundSLongLoops} and~\eqref{eq:LowerBoundSLongLOops} given \eqref{eq:asylongloops}:}
We need to put $\M=k$, noting that $\Nl_\L={\rm N}_\L^{\sqsup{\Nmin+1,N}}={\rm N}_\L^{\sqsup{\Nmin+1,k}}$ on the event $\{\Nl_\L=k\}$, and that
\begin{equation}
    \PPP^{\ssup{\bc,N}}_{\L_N}\rk{\Nl_{\L_N}=k}
=\PPP^{\ssup{\bc,N}}_{\L_N}\rk{{\rm N}_\L^{\sqsup{\Nmin+1,k}}=k}\ex^{-\sum_{r=k+1}^N\frac 1r \q^{\ssup \L}_r}\, ,
\end{equation}
since all $X_r$ with $r>k$ are zero on the event $\{\Nl_\L=k\}$. Now~\eqref{eq:LowerBoundSLongLOops} follows from~\eqref{eq:asylongloops}, noting that the last term is asymptotic to $k/N$ in the case $\lambda_1=0$ and to 1 in the case $\lambda_1>0$; see \eqref{Tsumasy}.

For deriving~\eqref{eq:UpperBoundSLongLoops} use, for any $r\in\{1,\dots,N\}$, just the estimate $\q_{r}\leq C \ex^{-\lambda_1\beta r\al^{-2/d}}$ (see Lemma \ref{CorollaryQkSupercritical}(iii)) in \eqref{longloopsdistribution}, then \eqref{eq:UpperBoundSLongLoops} follows, since $\sum_{r=\Nmin+1}^{\M} \partition_r\leq \frac 1{\Nmin+1} \sum_{r=\Nmin+1}^{\M} r \partition_r\leq k/(\Nmin+1)$. Indeed, also using \eqref{Equation731233},
\begin{equation}
\begin{split}
\PPP^{\ssup\bc}_{\L_N}\rk{\Nl_{\L_N}=k}&\leq C^{k/\Nmin} \ex^{-\lambda_1\beta k\al^{-2/d}}\ex^{-\sum_{r=\Nmin+1}^N\frac 1r (\q^{\ssup \L}_r-1)}\P\rk{\sum_{r=\Nmin+1}^{\M} rY_r=k }\\
&\leq C^{k/\Nmin} \ex^{-\lambda_1\beta k\al^{-2/d}}\, , 
\end{split}
\end{equation}
since $\q_{r}\geq 1$, according to Lemma~\ref{CorollaryQkSupercritical}(iii).

\textbf{Proof of upper bound in~\eqref{Prop28(1)}:} For deriving the upper bound, we need to use the stronger upper bound in Lemma~\ref{CorollaryQkSupercritical}(iii) and need to show that the difference to the lower bound is negligible. Indeed, we have, for some $c\in(0,\infty)$, and all large $L$,
\begin{equation}\label{Equation317232}
\begin{split}
    \sum_{\partition\in\Pfrak_k^{\sqsup{\Nmin+1,\M}}}&\prod_{r=\Nmin+1}^{\M}\frac{(\q_r^{\ssup\L})^{\partition_r}}{\partition_r!\,r^{\partition_r}}
    \le \ex^{- \lambda_1 \beta k\al^{-2/d}}\sum_{\partition\in\Pfrak_k^{\sqsup{\Nmin+1,\M}}}\prod_{r=\Nmin+1}^{\M}\frac{\rk{1+\ex^{-cr L^{-2}}}^{\partition_r}}{\partition_r!r^{\partition_r}}\\
    &\sim \ex^{-\lambda_1\beta k\al^{-2/d}}\frac{\M}{\Nmin}\E\ek{\prod_{r=\Nmin+1}^{\M}\rk{1+\ex^{-c rL^{-2} }}^{Y_r}\1\Big\{\sum_{r=\Nmin+1}^{\M} rY_r=k\Big\}}\\
    &\leq \ex^{-\lambda_1\beta k\al^{-2/d}}\frac{\M}{\Nmin}\E\ek{\exp\rk{\sum_{r=\Nmin+1}^{\M} Y_r \ex^{-c rL^{-2}}}\1\Big\{\sum_{r=\Nmin+1}^{\M} rY_r=k\Big\}}\, ,
\end{split}
\end{equation}
where the asymptotics uses that $\sum_{r=\Nmin+1}^\M\frac 1r\sim \frac \M\Nmin$, and the last estimate uses the bound $1+x\leq \ex^x$ for $x>0$. 

Note that for $r\ge N^+$, we have $\ex^{-c rL^{-2}}\leq \ex^{-c\log^2(N)}$ and hence the sum over $r\ge N^+$ is negligible:
\begin{equation}
  \sum_{r=N^+}^\M Y_r \ex^{-c rL^{-2}}
\leq \sum_{r\geq \Nmax}Y_r \ex^{-c \log^2N}\leq N  \ex^{-c \log^2N}=o(1)\, .
\end{equation}
According to \eqref{LogBookThm4.13}, it suffices to show that
\begin{equation}\label{Equation31723}
   \E\ek{\ex^{\sum_{r=\Nmin+1}^{\Nmax}Y_r\ex^{-crL^{-2}}}\1\Big\{\sum_{r=\Nmin+1}^{\M} rY_r=k\Big\}}\leq  \P\rk{\sum_{r=\Nmin+1}^{\M}rY_r=k}\big(1+o(1)\big)\, .
\end{equation} 
Note that by the independence of the Poisson point process variables,
\begin{equation}\label{Equation7282}
   \mbox{l.h.s.~of~\eqref{Equation31723}} =\sum_{l=0}^k\E\ek{\ex^{\sum_{r=\Nmin}^{\Nmax} Y_r\ex^{-crL^{-2}}}\1\Big\{\sum_{r=\Nmin}^{\Nmax}rY_r=l\Big\}}\P\!\rk{\sum_{r=\Nmax+1}^{\M}rY_r=k-l}\, .
\end{equation}
Note that, on the event $\{\sum_{r=\Nmin+1}^{\Nmax}rY_r=l\}$, we can always estimate
\begin{equation}\label{Ysumesti}
     \sum_{r=\Nmin+1}^{\Nmax} Y_r\ex^{-crL^{-2}}\le  \ex^{-c \Nmin L^{-2}}\sum_{r=\Nmin+1}^{\Nmax} \frac r \Nmin Y_r=  \frac{l}{\Nmin }\ex^{-c \log^{1/2}(N)}\, .
\end{equation}
We distinguish two cases, depending on the size of $l$ in the above sum.

\noindent\textbf{Sum on $l\le L^2\log^8(N)$:} Here, the bound in~\eqref{Ysumesti} is $o(1)$, such that we obtain
\begin{multline}\label{Equation111231}
    \sum_{l=0}^{N^{2/d}\log^8(N)}\E\ek{\ex^{\sum_{r=\Nmin+1}^{\Nmax}Y_r\ex^{-crn^{-2/d}}}\1\Big\{\sum_{r=\Nmin+1}^{\Nmax}rY_r=l\}}\P\rk{\sum_{r=\Nmax+1}^{\M}rY_r=k-l}\\
    \leq (1+o(1))\P\rk{\sum_{r=\Nmin+1}^{\M}rY_r=k}.
\end{multline}

\noindent\textbf{Sum on $l>L^{2}\log^8(N)$:} We derive an upper bound for any $l$, which shows that the sum on $l>L^{2}\log^8(N)$ is negligible.
 
Using the rough bound in~\eqref{Ysumesti}, we obtain
\begin{equation}\label{Yestilargel}
\begin{split}
&\E\ek{\ex^{\sum_{r=\Nmin+1}^{\Nmax} Y_r\ex^{-crL^{-2}}}\1\Big\{\sum_{r=\Nmin+1}^{\Nmax}rY_r=l\Big\}}\\
&\leq \exp\rk{\frac{l}{{\Nmin}}\ex^{-c\log^{1/2}(N)}}\P\rk{\sum_{r=\Nmin+1}^{\Nmax}rY_r=l}.
\end{split} 
\end{equation}
Furthermore, use the exponential Chebyshev inequality to see, for any $l\in\N$, and any $\alpha \in(0,\infty)$, estimating $\q_r\leq 2$ for $r\geq \Nmin+1$, 
\begin{equation}
    \P\rk{\sum_{r=\Nmin+1}^{\Nmax}rY_r=  l }\le \ex^{-\alpha l} \exp\Big(\sum_{r=\Nmin+1}^{\Nmax}\frac{2}{r}\rk{\ex^{\alpha r}-1}\Big)\, \le \ex^{-\frac l\Nmax} \ex^{4 \Nmax \alpha}\le \ex^{-\frac l\Nmax}\ex^4 \,,
\end{equation}
choosing $\alpha=1/\Nmax$ in the second step, and using that $\ex^{\alpha r}-1\leq 2\alpha r$ in the summation. Using this in \eqref{Yestilargel}, gives, for all large $N$,
\begin{equation}
\begin{split}
    \E\ek{\ex^{\sum_{r=\Nmin+1}^{\Nmax}Y_r\ex^{-crL^{-2}}}\1\Big\{\sum_{r=\Nmin+1}^{\Nmax}rY_r=l\Big\}}
    &\le \exp\Big( - \frac l\Nmax \big[1- \log^{3/2}(N)\ex^{-c\log^{1/2}(N)}\big]\Big)\ex^4\\
    &\leq \ex^{-\frac 12 \frac l\Nmax} \, .
\end{split}
\end{equation}

Now we sum on $l>L^2\log^8(N)$ and obtain, for all sufficiently large $N$,
\begin{equation}
\begin{aligned}
\sum_{l=L^2\log^8(N)}^k&\E\ek{\ex^{\sum_{r=\Nmin+1}^{\Nmax}Y_r\ex^{-crn^{-2/d}}}\1\Big\{\sum_{r=\Nmin+1}^{\Nmax}rY_r=l\}}\P\rk{\sum_{r=\Nmax+1}^{\M}rY_r=k-l}\\
    &\leq \sum_{l=L^2\log^8(N)}^{k}\ex^{-\frac 12 \frac l\Nmax}\leq k \ex^{-\frac 12 \log^6(N)}\leq o(\smfrac 1k).
\end{aligned}
    \end{equation}
     Now we build the sum of this with~\eqref{Equation111231} and obtain that the left-hand side of~\eqref{Equation31723} is not larger than $\P(\sum_{r=\Nmin+1}^{\M} r Y_r=k)(1+o(1))$, since this is itself asymptotic to $p_1(y)/\M\asymp \frac 1k$, according to~\eqref{LogBookThm4.13}. This concludes the proof of~\eqref{Equation31723}, and hence of the proposition.
\end{proof}

\subsection{Proof of Proposition~\ref{prop:ODLRO_sup}: asymptotics of the kernel}\label{sec-proofProp2.1}

Fix a boundary condition $\bc$ (diffusive or periodic) and drop `bc' from the notation. We abbreviate $\L$ for $\L_N$ and $L$ for $L_N$. Furthermore, we write ${\tt P}_\L$ instead of ${\tt P}_\L^{\ssup{N,{\rm bc}}}$ for the probability with respect to the Poisson point process $(X_r)_{r=1,\dots,N}$ introduced in Section~\ref{sec-PPP}.

Our goal is to show that there exists $c\in(0,\infty)$ such that, for any $\kappa, M>0$ and for any $x,y\in \L$, in the limit as $N\to\infty$,
\begin{equation}\label{EquationMainStatementRepeat}
\begin{split}
    (1&+o(1))(\rho-\rho_{\rm c})\phi_1(\smfrac{x}{L})\phi_1(\smfrac{y}{L})\\
    &\leq \sum_{r=1}^{N}g_{r\beta}^{\ssup{\L}}(x,y)
    \frac{\PPP_{\L}\left(\Partn_{\L}=N-r\right)}{\PPP_{\L}\left(\Partn_{\L}=N\right)}\\
    &\leq (1+o(1))(\rho-\rho_{\rm c})\Big(\phi_1(\smfrac{x}{L})\phi_1(\smfrac{y}{L})+\ex^{-M c}\Big)+C_M \psi(|x-y|)+ \ex^{-\kappa L^{d-2}}\, ,
\end{split}
\end{equation}
where 
$C_M\in(0,\infty)$ depends only on $M$, and $\psi$ satisfies $\psi(r)\leq r^{2-d}$ in the limit $r\to\infty$. Note that the middle term of~\eqref{EquationMainStatementRepeat} is equal to $\gamma_N^{\ssup{\L}}(x,y)$ by Corollary~\ref{cor-FKPPPrepr}. Hence,~\eqref{EquationMainStatementRepeat} implies Proposition~\ref{prop:ODLRO_sup}.

Recall that $\Nmin=\lceil L^2\log^{1/2}(N)\rceil$, that $\Nl_\L={\rm N}_\L^{\sqsup{\Nmin+1,N}}=\sum_{r=\Nmin+1}^{N} rX_r$ is the number of particles in long loops, and $\Ns_\L={\rm N}_\L-\Nl_\L$ is the number of particles in short loops. Under $\PPP_\L$, the number of particles in short and long loops are independent, hence we can build a convolution for ${\rm N}_\L= \Ns_\L+ \Nl_\L$ to obtain, for the probability term in the numerator,
\begin{equation}\label{Dirlowbound1}
\begin{aligned}
\PPP_\L\rk{\Partn_\L=N-r}=\sum_{k=0}^N\PPP_\L\rk{\Ns_\L=k}\PPP_\L\rk{\Nl_\L=N-k-r}\, ,\quad r\in\{1,\dots,N-k\}\, .
\end{aligned}
\end{equation}
Noting that $\Nl_\L$ takes values only in $\{0\}\cup\{\Nmin+1,\dots,N\}$, the $k$-sum can be restricted to $k\in\{N-r\}\cup\{\Nmin+1-r,\dots,N-r\}$.
We now apply twice~\eqref{eq:LowerBoundSLongLOops} in Proposition \ref{prop:MedLongLoops} to get, for $k$ and $r$ such that $r\leq N-k-\Nmax$, in the case $\l_1>0$,
\begin{equation}\label{propkrappli}
    \PPP_\L\rk{\Nl_\L=N-k-r}\sim \ex^{-\lambda_1\beta\rk{N-k-r}|\L|^{1-\frac 2d}}\frac{\ex^{\gamma}}{\Nmin}\sim  \ex^{\lambda_1\beta \,r|\L|^{1-\frac 2d}} \PPP_\L\rk{\Nl_\L=N-k}\, ,
\end{equation}
and in the case $\l_1=0$,
\begin{equation}\label{propkrappli_b}
    \PPP_\L\rk{\Nl_\L=N-k-r}\sim \frac{\ex^{\gamma}}{N}\sim  \PPP_\L\rk{\Nl_\L=N-k}\, . 
\end{equation}
Hence, in both cases, we have \eqref{propkrappli}.
We write the numerator on the left-hand side of \eqref{EquationMainStatementRepeat} as 
\begin{equation}\label{ProofProp1}
\begin{aligned}
    \sum_{r=1}^N &g^{\ssup\L}_{\beta r}(x,y)\PPP_\L\rk{\Partn_\L=N-r}\\
    &=  \sum_{(k,r)\in A}g^{\ssup\L}_{\beta r}(x,y)\PPP_\L\rk{\Ns_\L=k}\PPP_\L\rk{\Nl_\L=N-k-r}\, ,
\end{aligned}
\end{equation}
where $A=\{(k,r)\in\N_0\times\N\colon k+r\in\{N\}\cup\{1,\dots,N-\Nmin\}\}$. For some (large) $M\in(0,\infty)$, we now split this sum into three sums, $(I), (II)$, and $(III)$, on
\begin{eqnarray*}
    A_1&=&\{(k,r)\colon r\geq ML^2, k+r\leq N-\Nmax\}\, ,\\
    A_2&=&\{(k,r)\colon r< ML^2, k+r\in\{0,1,\dots, N-\Nmax\}\cup\{N\}\}\, , \\
    A_3&=&\{(k,r)\colon  N-\Nmax<k+r\leq N-\Nmin-1\}\, ,
\end{eqnarray*}
respectively.
From Lemma~\ref{CorollaryQkSupercritical}(i), we obtain that there is some $c\in(0,1)$ such that, for any $M>0$ and any large $N$,
\begin{equation}
    g^{\ssup \L}_{\beta r}(x,y)\leq  \ex^{-\l_1 r\beta L^{-2}}\frac 1{\al}\Big(\phi_1(\smfrac xL)\phi_1(\smfrac yL)+\Ocal(\ex^{-Mc})\Big),\qquad r\geq M L^2, x,y\in\L \, .
\end{equation}
Hence, on $A_1$ we can use this, the convolution in~\eqref{Dirlowbound1}, and~\eqref{propkrappli}, to obtain
\begin{equation}\label{ProofProp1I}
\begin{aligned}
    (I)&=\sum_{\substack{k\in\N_0,r\geq M L^2\colon\\ k+r\leq N-\Nmax}}g^{\ssup\L}_{\beta r}(x,y)\PPP_\L\rk{\Ns_\L=k}\PPP_\L\rk{\Nl_\L=N-k-r}\\
    &\leq \frac {1+o(1))}{\al}\Big( \phi_1(\smfrac xL)\phi_1(\smfrac yL)+\Ocal(\ex^{-Mc})\Big)\\
    &\qquad\qquad\times \sum_{k=0}^N \PPP_\L\rk{\Ns_\L=k}\PPP_\L\rk{\Nl_\L=N-k} \rk{N-\Nmin-k-M L^2}\, .
\end{aligned}
\end{equation}
We need to show that the sum on $k$ on the right-hand side of \eqref{ProofProp1I} is close to $\abs{\L}(\rho-\rhoc)\PPP_{\L}(\Partn_{\L}=N)$. It is clear that we can replace $N-\Nmin-k-M L^2$ by $N-k$. Recalling that $N=\rho |\L|$, and inserting the factor $1/\abs{\L}$ into the sum, we use again the convolution to see that, for any $\e>0$, we have
\begin{equation}
\begin{aligned}
    \Big|\sum_{k=0}^{N}&  \frac{N-k}{|\L|} \PPP_\L\rk{\Ns_\L=k} \PPP_\L\rk{\Nl_\L=N-k}-(\rho-\rhoc)\PPP_{\L}(\Partn_{\L}=N)\Big|\\
    &\leq \sum_{k=0}^{N}\PPP_\L\rk{\Ns_\L=k} \PPP_\L\rk{\Nl_\L=N-k}\big|\rhoc-\smfrac k{|\L|}\big|\\
    &\leq \e \sum_{k=-\e\al}^{\e\al} \PPP_\L\rk{\Ns_\L=k} \PPP_\L\rk{\Nl_\L=N-k}\\
    &\qquad +(\rhoc+\rho)\PPP_\L\Big(\Big|\frac 1{|\L|}\Ns_\L-\rhoc\Big|\geq\e\Big)\\
    &\leq \e \PPP_{\L}(\Partn_{\L}=N)+{\rm e}^{-\kappa L^{d-2}}\, ,
\end{aligned}
\end{equation}
with any $\kappa>0$, for all sufficiently large $N$, according to Corollary~\ref{Cor:shortloopsdev}. 

If we pick $\kappa$ larger than $(\rho-\rhoc)\beta\lambda_1$, Lemma~\ref{PropsotionUpperBoundTotal} tells us that the last summand is $o(\PPP_{\L}(\Partn_{\L}=N))$ as $N\to\infty$. Hence,~\eqref{ProofProp1I} implies that 
\begin{equation}
    (I)\leq \Big( \phi_1(\smfrac xL)\phi_1(\smfrac yL)+\Ocal(\ex^{-Mc})\Big)(\rho-\rhoc)\PPP_{\L}(\Partn_{\L}=N)(1+o(1)).
\end{equation} 
This yields the first term on the right-hand side of~\eqref{EquationMainStatementRepeat}. 

The lower bound for $\gamma_N^{\ssup{\L}}(x,y)$ in~\eqref{EquationMainStatementRepeat} follows from the preceding, using the lower bound for $g_{\beta r}(x,y)$ in Lemma~\ref{CorollaryQkSupercritical}(i), since $(II)$ and $(III)$ are non-negative.

Next, we handle the term $(II)$. For $r\leq M L^2$, we use the bound from Lemma~\ref{CorollaryQkSupercritical}(ii), with some suitable $C_M\in(0,\infty)$ and $c\in(0,1)$,
\begin{equation}
    g^{\ssup \L}_{\beta r}(x,y)\leq C_M \ex^{-\l_1 r\beta L^{-2}}  r^{-d/2}\ex^{-|x-y|^2/2cr},\qquad r\leq M L^2, x,y\in\L \, .
\end{equation}
Using first the convolution in~\eqref{Dirlowbound1}, then~\eqref{propkrappli}, and finally once more the convolution, we find
\begin{equation}\label{ProofProp1III}
\begin{aligned}
    (II)&=\sum_{\substack{k\in\N_0,r\leq M L^2\colon\\ k+r\leq N-\Nmax\mbox{ or }k+r=N}}g^{\ssup\L}_{\beta r}(x,y)\PPP_\L\rk{\Ns_\L=k}\PPP_\L\rk{\Nl_\L=N-k-r}\\
    &\leq C_M(1+o(1))\sum_{k=0}^N \PPP_\L\rk{\Ns_\L=k}\PPP_\L\rk{\Nl_\L=N-k}\sum_{r=1} ^{ML^2}r^{-d/2}\ex^{-|x-y|^2/2cr}\\
    &\leq C_M\PPP_{\L}\left(\Partn_{\L}=N\right)\psi(|x-y|)\, ,
\end{aligned}
\end{equation}
where $\psi$ does not depend on $M$ and satisfies $\psi(s)\leq C s^{2-d}$ as $s\to\infty$ for some $C\in (0,\infty$, as follows from the asymptotics (see~\cite[Lemma 4.3.2]{lawler2010random})
\begin{equation}\label{Greenasy}
    \sum_{k=1}^\infty k^{-d/2}\ex^{-r/k}\sim \frac{\Gamma(d/2-1)}{r^{d/2-1}}\qquad\text{as } r\to\infty\, .
\end{equation}
Hence, $(II)$ gives the second term on the right-hand side of~\eqref{EquationMainStatementRepeat}.

Finally, we handle the term $(III)$. We want to show that
\begin{equation}\label{ProofProp1IVa}
\begin{split}
    \frac{(III)}{\PPP_\L(\Partn_{\L}=N)} &= \hspace{-1em} \sum_{\substack{k\in\N_0,r\in\N\colon\\ N-\Nmax<k+r\leq N-\Nmin-1}}\hspace{-1em}
    g^{\ssup\L}_{\beta r}(x,y) \frac{\PPP_\L\rk{\Ns_\L=k}\PPP_\L\rk{\Nl_\L=N-k-r}}{\PPP_\L(\Partn_{\L}=N)} \\
    &\overset{!}{\leq} o(1)\, .
\end{split}
\end{equation}
In the case $\lambda_1>0$, we use that $\PPP_\L(\Partn_{\L} = N)\geq \ex^{-\l_1\beta NL^{-2}\rk{1+o(1)}}$ from Lemma~\ref{PropsotionUpperBoundTotal}, $\PPP_\L(\Nl_\L = k) \leq C^{k/\Nmin}\ex^{-\l_1\beta k L^{-2}}$ from Proposition~\ref{prop:MedLongLoops}, and $g_{\beta r}\leq C\, \ex^{-\l_1\beta rL^{-2}}$ from Lemma~\ref{CorollaryQkSupercritical}(i). This yields
\begin{equation}\label{ProofProp1IVb}
\begin{split}
    &\frac{(III)}{\PPP_\L(\Partn_{\L}=N)}\\
    &\quad \leq C \ex^{\l_1\beta NL^{-2}\rk{1+o(1)}} \hspace{-1em}
    \sum_{\substack{k\in\N_0,r\in\N\colon\\ N-\Nmax<k+r\leq N-\Nmin-1}}\hspace{-1.5em} 
    \ex^{-\l_1\beta rL^{-2}} C^{k/\Nmin} \ex^{-\l_1\beta kL^{-2}}\PPP(\Ns = N-r-k)\\
    &\quad \leq  C \ex^{\l_1\beta NL^{-2}\rk{1+o(1)}} \sum_{m=N-\Nmax}^{N-\Nmin-1} \ex^{-\l_1\beta mL^{-2}}\PPP(\Ns = N-m)\sum_{k=0}^m C^{k/\Nmin}\, .
\end{split}   
\end{equation}
Since $\sum_{k=0}^m C^{k/\Nmin} = \frac{C^{\frac{m+1}{\Nmin}}-1}{
C^{\frac{1}{\Nmin}}-1} \leq \Ocal(\Nmin) C^{\frac N\Nmin}=\ex^{o(L^{d-2})}$,
\begin{equation}\label{ProofProp1IVc}
\begin{split}
    \mbox{l.h.s.~of \eqref{ProofProp1IVb}}&\leq C \ex^{\l_1\beta NL^{-2}\rk{1+o(1)}} \ex^{o(L^{d-2})} \sum_{j=\Nmin}^{\Nmax} \ex^{-\l_1\beta (N-j)L^{-2}}\PPP(\Ns =j)\\
    &\leq C \ex^{o(L^{d-2})} \ex^{\l_1\beta\Nmax L^{-2}}\ex^{-\kappa L^{d-2}}\, ,
\end{split}   
\end{equation}
where in the last step we used Corollary~\ref{Cor:shortloopsdev}. After adapting the value of $\kappa$, this gives the third term on the right-hand side of~\eqref{EquationMainStatementRepeat} in the case $\lambda_1>0$.

For $\l_1=0$, we use that $\PPP_\L(\Partn_{\L} = N)\geq N^{-c}$, for some $c>0$, from Lemma~\ref{PropsotionUpperBoundTotal}, and $g_{\beta r}\leq C$ from Lemma~\ref{CorollaryQkSupercritical}(i). The claim then follows again from Corollary~\ref{Cor:shortloopsdev}:
\begin{equation}\label{ProofProp1IV_0}
\begin{split}
    \frac{(III)}{\PPP_\L(\Partn_{\L}=N)} &\leq C N^{c} \hspace{-1em}
    \sum_{\substack{k\in\N_0,r\in\N\colon\\ N-\Nmax<k+r\leq N-\Nmin-1}}\hspace{-1.5em}  \PPP(\Ns = N-r-k)\leq  C N^{c+1} \ex^{-\kappa L^{d-2}}\, .
\end{split}   
\end{equation}
After adapting the value of $\kappa$, this gives the third term on the right-hand side of~\eqref{EquationMainStatementRepeat} in the case $\lambda_1=0$, finishing the proof of Proposition~\ref{prop:ODLRO_sup}.

\subsection{Proof of Theorem~\ref{thm:ODLRO}(i) for diffusive and periodic boundary conditions}\label{sec-proofODLROnonfree}

Let us finally show that ODLRO follows from Proposition \ref{prop:ODLRO_sup}, i.e., let us prove Theorem \ref{thm:ODLRO}(i) for diffusive and periodic boundary conditions. Explicitly, we will show that the principal $L^2$-eigenvalue of $\Gamma_{N}^{\ssup{\L_N,{\rm bc}}}$ behaves like
\begin{equation}
\sup_{f\in L^2(\L_N)\colon \|f\|_{L^2(\L_N)}=1}\langle f,\Gamma_{N}^{\ssup{\L_N,{\rm bc}}} (f)\rangle \sim (\rho-\rhoc)|\L_N|,\qquad N \to\infty\, ,
\end{equation}
where $\L_N=L_N U=[-\frac 12 L_N,\frac 12 L_N]^d$.
For deriving the lower bound, we may use the $L^2$-normalised test function $f_N(x) =L_N^{-d/2}\phi_1^{\ssup{\bc}}(x/L_N)$, and obtain, as $N\to\infty$, the following lower bound from the lower bound in Proposition~\ref{prop:ODLRO_sup}:
\begin{equation}
\begin{aligned}
\langle f_N,\Gamma_{N}^{\ssup{\L_N,{\rm bc}}} (f_N)\rangle
&\geq  \int_{\L_N}\int_{\L_N} f_N(x)\gammabc{\bc}(x,y)f_N(y)\,\d x\d y\\
&\geq (1+o(1))\frac {\rho-\rhoc}{|\L_N|} \int_{\L_N}\int_{\L_N}\phi_1^{\ssup{\bc}}(\smfrac x{L_N})^2\phi_1^{\ssup{\bc}}(\smfrac y{L_N})^2\,\d x\d y\\
&=(1+o(1))(\rho-\rhoc) |\L_N|\, .
\end{aligned}
\end{equation}
Now we prove the upper bound. For any $L^2(\L_N)$-normalised function $f_N$, from the upper bound in Proposition~\ref{prop:ODLRO_sup}, we get that
\begin{equation}
\begin{aligned}
    &\langle f_N,\Gamma_{N}^{\ssup{\L_N,{\rm bc}}} (f_N)\rangle
    \le \int_{\L_N}\int_{\L_N} \abs{f_N(x)\gammabc{\bc}(x,y)f_N(y)}\,\d x\d y\\
    &\quad \leq (1+o(1))(\rho-\rhoc)\int_{\L_N}\int_{\L_N} \abs{f_N(x)\phi_1^{\ssup{\bc}}(\smfrac x{L_N})f_N(y)\phi_1^{\ssup{\bc}}(\smfrac y{L_N})}\,\d x\d y \\
    &\qquad +\Big(\ex^{-M c}+\ex^{ -\kappa L^{d-2}}\Big)\,\|f_N\|_1^2+ C_M \int_{\L_N}\int_{\L_N} |f_N(x)\psi(|x-y|)f_N(y)|\,\d x\d y\\
    &\quad \leq (1+o(1))(\rho-\rhoc)|\L_N|+o(1)|\L_N|+C_M \int_{\L_N}|x|^{2-d}\,\d x\\
    &\quad \leq (1+o(1))(\rho-\rhoc)|\L_N|+\Ocal(L_N^2)\leq (1+o(1))(\rho-\rhoc)|\L_N|\, ,
\end{aligned}
\end{equation}
after using the Cauchy--Schwarz inequality twice for both the first and second term, and using that $x\mapsto L_N^{-d/2}\phi_1^{\ssup{\bc}}(x/L_N)$ is $L^2(\L_N)$-normalised. For the last term, we used Young's inequality:
\begin{equation}\label{Young}
\int_{\R^d}\int_{\R^d} g(x) h(|x-y|)g(y)\,\d x\d y\leq \|h\|_1, \qquad g\in L^2(\R^d), \|g\|_2=1, h\in L^1(\R^d)\, ,
\end{equation}
This concludes the proof of Theorem \ref{thm:ODLRO}(i) for diffusive and periodic boundary condition.

\subsection{Proof of Proposition~\ref{prop:ODLRO_free}}\label{sec:ODLROfree}

In this section, we prove Proposition~\ref{prop:ODLRO_free}, that is, the equivalent of Theorem~\ref{thm:ODLRO}(i) for free boundary conditions, i.e., for $\gfree_t(x,y)$ equal to the (free) Gaussian density with variance $t$. In this case, we are not able to use an eigenvalue expansion as in the other cases, however, we have other tools that are more explicit. 

We formulate our main statement about the behaviour of the kernel of the one-particle-reduced density matrix. Proposition~\ref{prop:ODLRO_free} then follows from this in the same way as in Section~\ref{sec-proofODLROnonfree}.

\begin{proposition}\label{prop:ODLRO_sup_free}
Fix $\rho\in (\rhoc,\infty)$ and consider the centred box $\L_N$ of volume $N/\rho$ for $N\in\N$. Then, uniformly in $x/L_N,y/L_N\in U$, as $N\to\infty$,
\begin{equation}
    \gammabc{\free}(x,y)\sim  \rho-\rho_{\rm c}+o(1) +\psi(|x-y|),
\end{equation} 
for some function $\psi\colon (0,\infty)\to(0,\infty)$ that satisfies $\psi(r)\leq\Ocal( r^{2-d})$ as $r\to\infty$.
\end{proposition}
The main tool in our proof is again a Poisson point process of loops, however, we will need an extended version of it that is more \lq spatial\rq. Indeed, we consider a marked PPP on $\L$ with marks in $\N$ (called loops). The intensity measure of the process of points with mark $k\in\N$ has constant Lebesgue density $x\mapsto g^{\free}_{\beta k}(x,x)=\frac 1k(2\pi\beta k)^{-d/2}$, and the entire marked PPP is the superposition of these PPPs over $k\in\N$. The number $\Nop_\L$ of Poisson points in this process is Poisson-distributed  with parameter $|\L| p^{\free}=|\L| (2\pi\beta)^{-d/2}\zeta(\frac d2 +1)$, i.e., it has the same distribution as $\sum_{r\in\N} X_r$ under the Poisson process introduced in Section~\ref{sec-PPP} with $N=\infty$ and free boundary condition. Furthermore, the sum over all marks in the PPP in $\L$ has the same distribution as $N_\L=\sum_{r=1}^\infty r X_r$ (again with $N=\infty$). We incorporate the extended PPP into the notation $\P_\L^{\ssup{{\rm \free},\infty}}=\P_\L^{\ssup{{\rm \free}}}$, which is consistent with the notation that we introduced in Section~\ref{sec-PPP}.

The difference in our approach now in the case of free boundary condition is that we will split the PPP into the i.i.d.~sum $\Nop_\L=\sum_{z\in\L\cap\Z^d} \Nop_{z+U}$ of numbers of  Poisson points in unit cells in $\L$. Each of these $\Nop_{z+U}$ is easily seen to have a Pareto distribution, i.e., the probability that it is equal to $j$ is given by $j^{-\frac d2 -1}/\zeta(\frac d2+1)$, for any $j\in\N$. Hence, the main probabilistic tool will be to find polynomial asymptotics for deviations of a random walk with this step distribution. Hence, we begin with a summary of some existing results on random walks with heavy-tailed steps.

\begin{proposition}\label{prop:decayfree}
Let $d\ge 3$. Suppose $\rk{Z_i}_i$ is a collection of i.i.d.~$\N$-valued random variables such that, for some $\alpha>0$,
\begin{equation}
    \P\rk{Z_1=j}\sim \alpha j^{-d/2-1}\ \text{ as } j\to \infty.
\end{equation}
Write $S_n=\sum_{i=1}^n Z_i$. Suppose that $\E\ek{Z_1}=a>0$. Then, for every $c\in(0,1)$ and any $k\in\N$ satisfying $cn\leq an+k\leq \frac 1 c n$,
\begin{equation}
    \P\rk{S_n=an+k}\sim n\P\rk{X_1=k}\ \text{ as } n\to \infty,
\end{equation}
Furthermore, there exists $\e>0$ and $C>0$ such that for all $k\ge n^{1-\e}$
\begin{equation}
     \P\rk{S_n=an+k}\le C n\P\rk{Z_1=k}\ \mbox{for all large }n\in\N.
\end{equation}
\end{proposition}
\begin{proof} See \cite[Theorem 2.4]{berger2019notes} for the case $d=3$, \cite[Corollary 2.1]{DDS} for $d=4$, \cite[Theorem 2]{doney1997one} for $d\ge 5$.
\end{proof}

From this, we deduce the following:
\begin{lemma}\label{cor:freeparticlenumber}
Fix $\rho>\rhoc$. Recall that $\L_N$ is the centred box of volume $N/\rho$, $N\in \N$. Then, locally uniformly in $w\in (\rhoc,\infty)$, as $N\to\infty$,
\begin{equation}
    \Pfree_{\L_N}\rk{\Partn_{\L_N}=w\alN}\sim \frac{1}{(w-\rhoc)^{d/2+1}\rk{2\pi\beta\alN}^{d/2}}\, .
\end{equation}
\end{lemma}

\begin{proof}
From the definition of the above marked PPP, $\Nop_{\L_N}$ is Poisson-distributed with parameter $\alN\presfree$. Using the Markov inequality for a suitable stretched-exponential function, one derives that
\begin{equation}
     \Pfree_{\L_N}\rk{\abs{\Nop_{\L_N}-\alN\presfree}>\alN^{5/6}}\le \ex^{-\alN^{1/2}}\, ,
\end{equation}
for all large enough $N$. Furthermore, note that, conditionally on $\{\Nop_{\L_N}=k\}$, the total length $\Partn_{\L_N}$ of the loops is given by the sum $S_k=\sum_{i=1}^k Z_i$ of i.i.d. random variables with distribution
\begin{equation}
    \P\rk{Z_i=j}=\frac{1}{\presfree}\frac{1}{j}\gfree_{\beta j}(x,x)=\frac{1}{\zeta(\frac d2+1)} j^{-\frac d2-1}\, .
\end{equation}
In particular, $a=\E(Z_1)=\zeta(\frac d2)/\zeta(\frac d2+1)$. Furthermore, recall that $\rho_c=\zeta(\frac d2)(2\pi\beta)^{-d/2}=a \presfree$ and therefore, for $k\in\N$ satisfying $|k- \alN\presfree|\leq \alN^{5/6}$,
\begin{equation}
    w|\L_N|=ak+(w-\rhoc)|\L_N|+ o(|\L_N|)\, .
\end{equation}
This means that we can apply Proposition \ref{prop:decayfree}, and obtain
\begin{equation}
\begin{aligned}
    \Pfree_{\L_N}&\rk{\Partn_{\L_N}=w\alN}
    =o(|\L_N|^{-d/2}) +\sum_{\substack{k\in\N\colon\\|k- \alN\presfree|\leq \alN^{5/6}}}  \Pfree_{\L_N}(\Nop_{\L_N}=k)\\
    &\qquad\qquad\qquad\qquad\qquad \times\P(S_k=ak+(w-\rhoc)|\L_N|+ o(|\L_N|)\\
    &\sim \sum_{k\in\N\colon |k- \alN\presfree|\leq \alN^{5/6}}  \Pfree_{\L_N}(\Nop_{\L_N}=k) k\P\rk{Z_1=(w-\rhoc)|L_N|}\\
    &\sim \alN\presfree\P\rk{Z_1=(w-\rhoc)\alN}
    \sim\frac{1}{(w-\rhoc)^{d/2+1}\rk{2\pi\beta\alN}^{d/2}} \, .
\end{aligned}
\end{equation}
\end{proof}

We are now ready to compute the asymptotics of the kernel for supercritical densities.

\begin{proof}[Proof of Proposition~\ref{prop:ODLRO_sup_free}]
We shorten $\L=\L_N$. The idea of the proof is as follows. We start from~\eqref{eq:rdm_bc} (which holds by definition also for free boundary conditions), and recall that $N=\rho |\L|$. We restrict the sum on $r\in\{1,\dots,N\}$ to a neighbourhood of $(\rho-\rhoc)|\L|$, and show that the remainder of the sum is negligible, since the expectation of ${\rm N}_{\L}$ is asymptotic to $\rhoc |\L|$. For $r\approx (\rho-\rhoc)|\L|$ we may use explicit formulas for the $g$-term (here the dependence on $x$ and $y$ vanishes), and Lemma~\ref{cor:freeparticlenumber} for the two probability terms.

We put $\rhoe=\rho-\rhoc\in(0,\infty)$ and observe that $\gfree_{\beta r}(x,y)=(2\pi r \beta)^{-d/2}\ex^{-|x-y|^2/2r\beta}$. Fix $\e>0$. Note that there exists $a_\e$ such that $a_\e\to 0$ as $\e\downarrow 0$ and for all $k$ with $\rk{\rhoe-\e}|\L| \le r\le \rk{\rhoe+\e}|\L|$,
\begin{equation}\label{gupplowbound}
\Big| \frac {\gfree_{\beta r}(x,y)}{(2\pi\rhoe\al \beta)^{-d/2}}-1\Big|\le a_\e,\qquad x,y\in \L .
\end{equation}
Indeed, since $\abs{x-y}^2\le L_N^2$, we have that $\abs{x-y}^2/r=o(1)$ uniformly in the $k$'s specified above. 
Note that $\rk{\Partn_{z+U}}_{z\in{\L_N}\cap\Z^d}$ is a family of i.i.d.~random variables with mean $\rhoc$ under $\Pfree_{\L_N}$. Hence, by the strong law of large numbers,
\begin{equation}\label{EquationSLLN}
\lim_{N\to\infty} \Pfree_{\L}\rk{\abs{\Partn_{\L}-\rhoc\al}>\e\al}=0\, .
\end{equation}
Using \eqref{gupplowbound}, we can bound
\begin{equation}
\begin{aligned}
    \sum_{\mycom{k\in \Z}{\abs{k-\rhoe\al}<\e\al}}&\gfree_{\beta k}(x,y)\Pfree_{\L}\rk{\Partn_{\L}=\rho\al-k}\\
    &\le (2\pi\rhoe\al \beta)^{-d/2}\rk{1+a_\e} \Pfree_{\L}\rk{\abs{\Partn_{\L}-\rhoc\al}<\e\alN}\\
    &\leq (2\pi\rhoe\al \beta)^{-d/2}\rk{1+a_\e}\, ,
\end{aligned}
\end{equation}
and an analogous lower bound, using \eqref{EquationSLLN}. Hence, employing Lemma~\ref{cor:freeparticlenumber},
\begin{equation}
    \sum_{\mycom{k\in \Z}{\abs{k-\rhoe\al}<\e\al}}\gfree_{\beta k}(x,y)\frac{\Pfree_\L\rk{\Partn_{\L}=\rho\al-k}}{\Pfree_\L\rk{\Partn_{\L}=\rho\al}}\le \rhoe\rk{1+o(1)+a_\e}\, , 
\end{equation}
and an analogous lower bound (cf.~\eqref{EquationSLLN}). Recalling $N=\rho|\L|$, making $N\to\infty$ and $\e\downarrow 0$, we see that this part of the sum has the claimed asymptotics. We conclude the proof by showing that the sum on $k\in \N$ with $\abs{k-\rhoe\al}>\e\al$ is negligible. We split this in the upper part and the lower part.

\textbf{Large $k$'s}: Since $\gfree_{\beta k}(x,y)\le \Ocal\rk{k^{-d/2}}$, we see that
\begin{equation}
\begin{aligned}
    \sum_{k=\rk{\rhoe+\e}\al}^{\rho\al}\gfree_{\beta k}(x,y)\frac{\Pfree_{\L}\rk{\Partn_{\L}=\rho|\L|-k}}{\Pfree_{\L}\rk{\Partn_{\L}=\rho|\L| }}
    &\le \Ocal\big(|\L|^{-d/2}\big)\frac{\Pfree_{\L}\rk{\Partn_{\L}<\rk{\rhoc-\e}|\L|}}{\Pfree_{\L}\rk{\Partn_{\L}=\rho|\L|}}\\
&    \leq\frac {o\big(|\L|^{-d/2}\big)}{\Pfree_{\L}\rk{\Partn_{\L}=\rho|\L|}}\leq o(1) \, ,
\end{aligned}
\end{equation}
using \eqref{EquationSLLN} in the first step and Lemma~\ref{cor:freeparticlenumber} in the last step.

\textbf{Small $k$'s}: For $k<\rk{\rhoe-\e}\al$, we have by Lemma~\ref{cor:freeparticlenumber} (applied to $w=\rho$ and to $w=\rho-k/|\L|$ with $b\geq \e$) that
\begin{equation}
    \frac{\Pfree_\L\rk{\Partn_{\L}=\rho\al-k}}{\Pfree_\L\rk{\Partn_{\L}=\rho\al}}\sim C(\rhoe-\smfrac k {|\L|})^{-1-d/2}\, ,
\end{equation}
for some $C$, depending only on $\rho$ and $\rhoc$. Hence, using $C$ as a generic positive constant that depends only on $\rho$ and $\rhoc$ and $d$, we have that for all $N\in\N$,
\begin{equation}
\begin{aligned}
    &\sum_{k=1}^{\rk{\rhoe-\e}|\L|}\gfree_{\beta k}(x,y)\frac{\Pfree_{\L}\rk{\Partn_{\L}=\rho|\L|-k}}{\Pfree_\L\rk{\Partn_{\L}=\rho\al}}\\
    &\quad \le C \sum_{k=1}^{\frac \rhoe 2 \al}\gfree_{\beta k}(x,y)(\rhoe/2)^{-1-d/2}+C\sum_{k=\frac \rhoe 2 \al}^{\rk{\rhoe-\e}|\L|}k^{-d/2} \e ^{-1-d/2}\\
    &\quad \leq  o(1)+C \sum_{k\in\N}\gfree_{\beta k}(x,y).
\end{aligned}
\end{equation}
The last term behaves as $|x-y|^{2-d}$ for $|x-y|\to\infty$, as follows from \eqref{Greenasy}.
\end{proof}

\section{No ODLRO in the subcritical regime}\label{sec:nonODLRO}

\noindent We now show that, for any particle density $\rho$ below the critical threshold $\rhoc$, there is no ODLRO, in any dimension $d$. Our main result in this respect is the validity of \eqref{eq:ODLRO_sub}, that is, the following.

\begin{proposition}\label{prop:ODLRO_sub}
Fix $\rho\in(0,\rhoc)$ and consider the centred box $\L_N$ of volume $N/\rho$, $N\in\N$. Then, for any boundary condition $\bc\in\gk{\diff,\per, \free}$, there exist constants $C,c\in(0,\infty)$, such that, for all sufficiently large $N$,
\begin{equation}
     \gamma_N^{\ssup{\L_N,{\bc}}}(x,y)\leq C\ex^{-c |x-y|}+C\ex^{-c|\L_N|^{1/2}},\qquad x,y\in\L_N\, .
\end{equation}
\end{proposition}
Indeed, since $x\mapsto \ex^{-c\abs{x}}$ is integrable, Young's inequality (see \eqref{Young}) directly implies that the principal eigenvalue of the operator $\Gamma^{\ssup{\L_N,{\rm bc}}}_N$ is even bounded in $N$, and this implies that the system does not exhibit ODLRO.

The idea of our proof of Proposition~\ref{prop:ODLRO_sub} is the following. We will make a change of measure of our crucial Poisson point process to a transformation that exponentially suppresses long loops. This process has very nice properties, since the arising sums can essentially be handled as if they were only over a bounded summation set. In particular, the number of particles will turn out to satisfy a local central limit theorem, allowing us to estimate the transformed partition function with high precision.

We now prepare to prove Proposition~\ref{prop:ODLRO_sub}.

We first introduce a transformation of the PPP $(X_k)_{k\in\N}$ introduced in Section~\ref{sec-PPP}, with an additional chemical potential $\mu<0$. The intensity measure of this PPP is given by
\begin{equation}
    \nu_{\L,\mu}^{\ssup{{\bc}}}=\sum_{k=1}^{N}\frac{\ex^{\beta\mu k}}{k}\q_k^{\ssup{\L,{\rm bc}}}\,\delta_k\, .
\end{equation}
We denote its distribution by $\Pbc_{\L,\mu}$. Then $\Pbc_{\L,\mu}$ has Radon--Nikodym derivative with respect to $\Pbc_{\L} = \Pbc_{\L,0}$ given by
\begin{equation}\label{Changeofdensity}
    \frac{\d\Pbc_{\L,\mu}}{\d\Pbc_{\L,0}}=\ex^{|\L|(\presbc_{\L,0}-\presbc_{\L,\mu})}\, \ex^{\beta\mu\Nop_{\L}}\, ,
\end{equation}
where we recall that $\Nop_{\L}=\sum_{k\in\N}  X_k$ is the number of Poisson points in the process, and the respective pressure and density in finite volume are given by
\begin{equation}
    \presbcN_{\L,\mu}=\frac{1}{\al}\sum_{k= 1}^N\frac{\ex^{\beta\mu k}}{k}\q_k^{\ssup{\L,{\rm bc}}}\qquad\text{and}\qquad \densbcN_{\L,\mu}=\frac{1}{\al}\sum_{k= 1}^N {\ex^{\beta\mu k}}\q_k^{\ssup{\L,{\rm bc}}}=\frac{1}{\beta}\frac{\d}{\d \mu }\presbcN_{\L,\mu}\, ,
\end{equation}
where $\q_k^{\ssup{\L,{\rm bc}}}=\int_\L\gbc_{\beta k}(x,x)\,\d x$.

Then $\mu\mapsto  \presbcN_{\L}(\mu)$ and $\mu\mapsto \rho_\L^{\ssup{\bc,N}}(\mu)$ are (approximate) power series that converge locally uniformly for $\mu\in(-\infty,0)$ in the thermodynamic limit as $N\to\infty$, coupled with $\L_N\to\R^d$. This is formulated as follows.

\begin{lemma}\label{ApproximationByThermodynamicFunctions}
For any boundary condition $\bc$, for the centred box $\L_N$ with volume $N/\rho$, as $N\to\infty$,
\begin{eqnarray}
      \lim_{N\to\infty}\presbcN_{\L_N,\mu}= \presfree(\mu)&=&\sum_{k\in\N}\frac{\ex^{\beta\mu k}}k(2\pi \beta k)^{-d/2},\\
       \lim_{N\to\infty}\densbcN_{\L,\mu}= \densfree(\mu)&=&\sum_{k\in\N}\ex^{\beta\mu k}(2\pi \beta k)^{-d/2}\, ,
\end{eqnarray}
locally uniformly for $\mu\in(-\infty,0)$.
The result also continues to hold for the $\mu$-derivatives. For the pressure $\presbc_{\L}(\mu)$, the convergence is locally uniform even in $(-\infty,0]$.
\end{lemma}
\begin{proof}
The proof is an easier version of the proof of \eqref{eq:MeanShort}. Indeed, recall that $\frac1{|\L_N|}\q_k^{\ssup{\L,{\rm bc}}}$ converges towards $g_{\beta k}(0,0)=(2\pi\beta k)^{-d/2}$, and observe that the factors $\frac 1k\ex^{\beta\mu k}$ provide a summable majorant.
\end{proof}
As a consequence, the quantity
\begin{equation}\label{murhodefN}
\mu^{\ssup{{\bc,N}}}_{\L_N}(\rho)\in(-\infty,0)\qquad\mbox{defined by}\qquad\densbcN_{\L}(\mu^{\ssup{{\bc,N}}}_{\L_N}(\rho)) =\rho,
\end{equation}
converges in the thermodynamic limit to the chemical potential $\mu=\mu(\rho)$ defined by $\densfree(\mu(\rho))=\rho$.

The good control on the pressure and the density from Lemma~\ref{ApproximationByThermodynamicFunctions} also gives us the tools to handle the limiting distribution of the particle number $\Partn_\L=\sum_{k=1}^N k X_k$ in the Poisson points (loops):
\begin{lemma}[Local CLT for the particle number]\label{LemmaCLT}
For any boundary condition $\bc$, for the centred box $\L_N$ with volume $N/\rho$, as $N\to\infty$, the particle number ${\rm N}_{\L_N}$ satisfies a local central limit theorem under $\PPP_{\L_N,\mu_N}^{\ssup{\bc,N}}$ 
In particular, there is $C\in (0,1)$ such that for any $N\in\N$ and any $k\in \{0,\ldots,\alN^{1/2}\}$,
\begin{equation}
    C^{-1}\alN^{-1/2} \le \Pbc_{\L_N,\mu_N}\rk{\Partn_{\L_N}=\rho\alN-k}\le C\alN^{-1/2}\, .
\end{equation}
\end{lemma}

\begin{proof}
We abbreviate $\L$ for $\L_N$. Note that ${\rm N}_{\L}=\sum_{z\in\L\cap\Z^d}{\rm N}_{z+U}$, where we recall that $U=[-\frac 12,\frac 12]^d$. Under $\P_{\L_N,\mu_N}^{\ssup{\rm bc}}$, the family $({\rm N}_{z+U})_{z\in\L\cap\Z^d}$ is an independent collection of random variables ${\rm N}_{z+U}$ (however, not identically distributed) with $r$-th moment given by
\begin{equation}
  \E_{\L_N,\mu_N}^{\ssup{\rm bc}}[ {\rm N}_{z+U}^r]= \sum_{k=1}^N{\ex^{\beta\mu k}}k^{r-1}\int_{z+U}g_{\beta j}^{\ssup{\L_N,{\rm bc}}}(x,x)\,\d x,\qquad r\in\N.
\end{equation}
Lemma~\ref{ApproximationByThermodynamicFunctions} gives the convergence of $\sum_{z\in\L_N\cap\Z^d} \E_{\L_N,\mu_N}^{\ssup{\rm bc}}[ {\rm N}_{z+U}^r]$ as $N\to\infty$ for any $r\in\N$ towards the values for the limiting objects for the free boundary condition, that is to $\sum_{k\in\N}\ex^{\beta\mu k}k^{r-1}(2\pi\beta k)^{-d/2}$. Now, the Lindeberg central limit theorem (see~\cite[Theorem 5.12]{kallenberg1997foundations}) implies the central limit theorem for ${\rm N}_{\L_N}$, that is, the distributional convergence of 
\begin{equation*}
\begin{split}
    &\frac{\sum_{z\in\L_N\cap\Z^d}[{\rm N}_{z+U}-\E_{\L_N,\mu_N}^{\ssup{\rm bc}}[ {\rm N}_{z+U}]]}
    {(\sum_{z\in\L_N\cap\Z^d}[\E_{\L_N,\mu_N}^{\ssup{\rm bc}}[ {\rm N}_{z+U}^2]- \E_{\L_N,\mu_N}^{\ssup{\rm bc}}[ {\rm N}_{z+U}]^2])^{1/2}}\\
    &\qquad\qquad =\frac{{\rm N}_{\L_N}-N\rho}{(\sum_{z\in\L_N\cap\Z^d}[\E_{\L_N,\mu_N}^{\ssup{\rm bc}}[ {\rm N}_{z+U}^2]- \E_{\L_N,\mu_N}^{\ssup{\rm bc}}[ {\rm N}_{z+U}]^2])^{1/2}}
\end{split}
\end{equation*}
towards a standard normal variable (recall \eqref{murhodefN}). Note that, again by Lemma~\ref{ApproximationByThermodynamicFunctions}, the denominator is asymptotic to $ C|\L_N|^{1/2}$ for some $C=C(\rho)\in(0,\infty)$ as $N\to\infty$.

Furthermore, as the third and fourth moments are bounded and converge uniformly (this follows from Lemma~\ref{ApproximationByThermodynamicFunctions}), we also have the local central limit theorem, i.e., that the density of the rescaled sum is uniformly close to a Gaussian density, see~\cite[Chapter VII, Theorem 26]{Cramer_1970}. This implies the last assertion.
\end{proof}

We can now finish the proof:
\begin{proof}[Proof of Proposition~\ref{prop:ODLRO_sub}]
For any $\mu\in(-\infty,0)$, we can make the following change of measure:
\begin{equation}\label{EquationCHangeOfMEasure}
    \frac{\Pbc_{\L,0}\rk{\Partn_{\L}=\rho\al-k}}{\Pbc_{\L,0}\rk{\Partn_{\L}=\rho\al}}=\ex^{k\beta\mu} \frac{\Pbc_{\L,\mu}\rk{\Partn_{\L}=\rho\al-k}}{\Pbc_{\L,\mu}\rk{\Partn_{\L}=\rho\al}}\, .
\end{equation}

Again, we write $\L$ for $\L_N$. We start from Corollary~\ref{cor-FKPPPrepr}, which reads, using~\eqref{EquationCHangeOfMEasure},
\begin{equation}\label{Equation6231}
      \gamma_N^{\ssup{\L,{\bc}}}(x,y)=\sum_{k=1}^{N}g_{\beta k}^{\ssup{\L,{\rm bc}}}(x,y)\ex^{k\beta\mu} \frac{\Pbc_{\L,\mu}\rk{\Partn_{\L}=\rho\al-k}}{\Pbc_{\L,\mu}\rk{\Partn_{\L}=\rho\al}}\, .
\end{equation}
We will use this for $\mu$ replaced by $\mu_N=\mu^{\ssup{{\rm bc}}}_{\L_N}(\rho)$ (see~\eqref{murhodefN}). For the sum on $k\leq \al^{1/2}$, we can use the local central limit theorem of Lemma~\ref{LemmaCLT} in both numerator and denominator, and for the remaining $k$, we use a crude bound on $g$ and on the numerator, but again the local CLT for the denominator. Let us turn to the details.

By Lemma~\ref{LemmaCLT},
\begin{equation}\label{uppboundNonODLRO}
    \sum_{k\leq \al^{1/2}}g_{\beta k}^{\ssup{\L,{\rm bc}}}(x,y)\ex^{k\beta\mu_N} \frac{\Pbc_{\L,\mu_N}\rk{\Partn_{\L}=\rho\al-k}}{\Pbc_{\L,\mu_N}\rk{\Partn_{\L}=\rho\al}}\le C^2\sum_{k\leq \al^{1/2}}g_{\beta k}^{\ssup{\L,{\rm bc}}}(x,y)\ex^{k\beta\mu_N}\, .
\end{equation}
Furthermore, using Lemma~\ref{CorollaryQkSupercritical}(ii) for the case of boundary conditions, for some $c>0$,
\begin{equation}
    g_{\beta k}^{\ssup{\L,{\rm bc}}}(x,y)\le \ex^{-\frac{{\abs{x-y}^2}}{2\beta c k}}, \qquad k\in\N, x,y\in\L\, .
\end{equation}
Hence, the right-hand side of~\eqref{uppboundNonODLRO} is not larger than $C^2\sum_{k\leq \al^{1/2}}\ex^{-\frac{{\abs{x-y}^2}}{2\beta k}}\ex^{k\beta\mu_N }$. Now reserve $\ex^{\frac 12 k\beta\mu_N}$ for convergence, and use a quadratic extension for the remaining terms in the exponent to see that, for sufficiently small $c'\in(0,\infty)$, any $N$, and all $x,y$,
\begin{equation}
\begin{split}
\frac{{\abs{x-y}^2}}{2\beta c k}&-\frac 12 k\beta\mu_N - c' \abs{x-y}\\
&=\frac 1k\Big(\frac{|x-y|}{\sqrt{2\beta c}}-k\frac 1{\sqrt 2}\sqrt{-\beta \mu_N}\Big)^2
+\abs{x-y} (\sqrt{-\mu_N/c}-c'))\geq 0\, .
\end{split}
\end{equation}
This implies that
\begin{equation}\label{firsttermNoODLRO}
\mbox{l.h.s.~of \eqref{uppboundNonODLRO}}\leq \sum_{k\leq \al^{1/2}}\ex^{\frac 12 k\beta\mu_N} \ex^{-c' |x-y|}\leq C\ex^{-c |x-y|}\, ,
\end{equation}
for some $C\in(0,\infty)$ that depends only on $\beta$ and the infimum of the sequence $(\mu_N)_N$. 

Let us turn to the sum on $k\ge \al^{1/2}$. Here we estimate $\al^{1/2}g_{\beta k}(x,y)\leq C$ for any $x,y\in\L$, using Lemma~\ref{CorollaryQkSupercritical}(i) for the case of (diffusive or periodic) boundary conditions. Using again the local CLT for the denominator and the simple bound $1$ for the probability, leads to the estimate
\begin{equation}
\begin{split}
     \sum_{k=\al^{1/2}}^{\rho\al}g_{\beta k}^{\ssup{\L,{\rm bc}}}(x,y)\ex^{k\beta\mu_N} \frac{\Pbc_{\L,\mu_N}\rk{\Partn_{\L}=\rho\al-k}}{\Pbc_{\L,\mu_N}\rk{\Partn_{\L}=\rho\al}}
     &\le C\sum_{k=\al^{1/2}}^{\rho\al}g_{\beta k}^{\ssup{\L,{\rm bc}}}(x,y)\al^{1/2} \ex^{k\beta\mu_N}\\
     &\le \sum_{k\geq \al^{1/2}} \ex^{k\beta\mu_N}\leq C \ex^{-c\al^{1/2}}\, ,
\end{split}
\end{equation}
where we have used the formula for the remainder sum of a geometric  series, for a suitably picked $C$ (larger than the maximum of $(1-\ex^{\beta \mu_N})^{-1}$ over $N$), and $c>0$ picked smaller than the infimum of $|\beta\mu_N|$ over $N$ (recall from Lemma~\ref{ApproximationByThermodynamicFunctions} that $\mu_N\to\mu\in(-\infty,0)$ defined by $\densfree(\mu)=\rho$). Since $|x-y|\leq \sqrt{|\L|}$, we see that the upper bound in~\eqref{firsttermNoODLRO} is not smaller than this bound. Hence, combining it with~\eqref{firsttermNoODLRO}, we have derived the claim, for suitable $C,c\in(0,\infty)$.
\end{proof}

\section{Distribution of the lengths of the long loops}\label{sec-prooflongloops}

In this section, we prove Proposition~\ref{prop:PoissonDirichlet}.
Fix $\rho>\rhoc$ for the whole section and write $\rhoe=\rho- \rhoc$.

We will show that, for $\bc\in \gk{{\rm diffusive},\per}$, the finite-dimensional distributions of $\frac 1{|\L_N|}(L^{\ssup N}_i)_{i\in\N}$ under $\PPP_{\L_N}(\cdot\,\vert\, \Partn_{\L_N}=N)$ converge to the ones of the Poisson--Dirichlet distribution of parameter $1$. The latter ones are known~\cite{arratia2003logarithmic} to have, for any $s\in\N$, the $s$-dimensional probability density
\begin{equation}
 f\hk{s}(x)=\frac{\ex^{\gamma}}{\prod_{i=1}^s x_i}p\rk{\frac{1-(x_1+\ldots+x_s)}{x_s}}, \qquad 0\leq x_s\leq \dots\leq x_1\leq 1, \sum_{i=1}^s x_i\leq 1\, ,
\end{equation}
where we recall the definition of $p\colon [0,\infty)\to[0,\infty)$ from the beginning of Section~\ref{sec-longloops} (see around~\eqref{LaplacetrafoPD}), in particular recall that $\gamma$ is the Euler--Mascheroni constant. Note that $f\hk{s}$ is continuous and bounded on its domain, since $p$ is.

We will prove converge in distribution via locally uniform convergence of densities in the interior of the domain of $f\hk{s}$. The main step in the proof of Proposition~\ref{prop:PoissonDirichlet} is the following.
\begin{proposition}\label{prop:convergenceOfDensitiesPoissonDirichlet}
Consider the centred box $\L_N$ with volume $N/\rho$. Fix $s\in\N$ and $0<x_s< x_{s-1}<\ldots<x_1<1$ with $0<x_1+\ldots+x_s<1$. Then, for all positive integers $m_s\le m_{s-1}\le \ldots\le  m_1$, depending on $N$, with $m_i/(\rhoe\alN)\to x_i$ for all $i\in\gk{1,\ldots,s}$ as $N\to\infty$, 
\begin{equation}
    \lim_{N\to\infty}(\rhoe\alN)^s\,\PPP_{\L_N}^{\ssup\bc}\Big(L^{\ssup N}_i=m_i,\, \forall i\in\gk{1,\ldots,s}\Big|\Partn_{\L_N}=N\Big)=f\hk{s}(x_1,\ldots,x_r)\, .
\end{equation}
\end{proposition}
From this, the weak convergence of $(\rhoe|\L_N|)^{-1}(L_i^{\ssup N})_{i=1,\dots,s}$ towards the first $s$-di\-men\-sion\-al distribution of the Poisson--Dirichlet distribution follows, according to the Portemanteau theorem. From Scheff\'{e}'s theorem, the convergence of the entire sequence $(\rhoe|\L_N|)^{-1}(L_i^{\ssup N})_{i\in\N}$ follows, that is, Proposition~\ref{prop:PoissonDirichlet} (see~\cite[Corollary 5.11]{arratia2003logarithmic}).

\begin{proof} We can take $N$ so large that $m_s<m_{s-1}<\ldots<m_1$, as the $x_i$'s are different. As usual, we write $\L=\L_N$, and drop the superscripts $\bc$ and $N$. 

Abbreviate $A=\{L^{\ssup N}_i=m_i,\, \forall i\in\gk{1,\ldots,s}\}$. For $N$ large enough, it is independent of $(X_r)_{r\leq  \Nmin}$. Hence, we can decompose as follows.
\begin{equation}
\begin{split}
    &\PPP_\L^{\ssup\bc}(A\,\vert\, \Partn_\L=N)\\
    &\qquad =\sum_{k\in\N\colon |\frac k{|\L|}-\rhoe|\leq \eps_N} \frac {\PPP_\L^{\ssup\bc}(\Ns=N-k)}{\PPP_\L^{\ssup\bc}(\Partn_\L=N)}\PPP_\L^{\ssup\bc}(A\cap \{\Nl=k\})+o(|\L_N|^s)\, ,
\end{split}
\end{equation}
where the sequence $(\eps_N)_N$ is picked in $(0,1)$ tending to zero such that the remaining sum on $k$ on $\PPP_\L^{\ssup\bc}(\Ns=N-k)/\PPP_\L^{\ssup\bc}(\Partn_\L=N)$ is $o(|\L_N|^s)$, using the last assertion of Corollary~\ref{Cor:shortloopsdev}.

Hence, it will be sufficient to show that, uniformly for all $k=k_N$ satisfying $|\frac k{|\L|}-\rhoe|\leq \eps_N$,
\begin{equation}
\lim_{N\to\infty}(\rhoe\al)^s\,\PPP_\L^{\ssup\bc}(A\,\vert\,\Nl=k)=f\hk{s}(x_1,\ldots,x_r)\, .
\end{equation}
We write the probability explicitly in terms of the Poisson point process $(X_r)_{r=\Nmin+1}^N$. We abbreviate $[j]=\{1,\dots,j\}$ for $j\in\N$, and we write $\overline m=\sum_{i\in[s]}m_i$. We can decompose into independent events as follows:
\begin{equation}
\begin{split}
    &A\cap \{\Nl=k\}\\
    &\quad =\Big(\bigcap_{i\in [s]}\{X_{m_i}=1\}\Big)\cap\Big(\bigcap_{r\in[N]\setminus ([m_s]\cup \{m_i\colon i\in[s]\})}\{X_r=0\}\Big)\cap \big\{{\rm N}_\L^{\sqsup{\Nmin+1,m_s-1}}=k-\overline m\big\}\, .
\end{split}
\end{equation}
Hence,
    \begin{equation}\label{011220231W}
        \begin{aligned}
 \PPP_\L^{\ssup\bc}(A\,\vert\,\Nl=k)&= 
    \ex^{-\sum_{r=m_s+1}^N\frac{1}{r}\qbc_r}
            \Big(\prod_{i\in [s]}\frac{\qbc_{m_i}}{m_i}\Big)\frac{\PPP_\L^{\ssup\bc}\rk{{\rm N}_\L^{\sqsup{\Nmin+1,m_s-1} }=k-\overline m}}{\PPP_\L^{\ssup\bc}\rk{\Nl=k}}\, .
        \end{aligned}
    \end{equation}
Now we carry out the limit as $N\to\infty$ and recall that $k/|\L|\to\rhoe$, $m_i/|\L|\to x_i \rhoe$ for $i\in[s]$ and hence $\overline m /|\L|\to \rhoe\sum_{i\in[s]}x_i$. We may then use~\eqref{eq:QkSupercritical} for all the $\qbc_{r}$, and Proposition~\ref{prop:MedLongLoops} for the two terms in the last quotient. We need to distinguish the cases $\lambda_1>0$ and $\lambda_1=0$. 

In the case $\l_1>0$, every $\qbc_{r}$ vanishes as $N\to\infty$ in a stretched-exponential way, such that the first term on the right-hand side of~\eqref{011220231W} tends to one, and the second is 
\begin{equation}
\prod_{i\in [s]}\frac{\qbc_{m_i}}{m_i}\sim \prod_{i\in [s]}\frac{\ex^{-\lambda_1 \beta m_iL^{-2}}}{x_i |\L| \rhoe} 
= (|\L|\rhoe)^{-s} \frac {\ex^{-\lambda_i \beta \overline m L^{-2}}}{\prod_{i\in [s]}x_i}\, .
\end{equation}
Now, recalling that $p(1)=\ex^{-\gamma}$, the last term on the right-hand side of~\eqref{011220231W} is
\begin{equation}
\begin{split}
\frac{\PPP_\L^{\ssup\bc}\rk{{\rm N}_\L^{\sqsup{\Nmin+1,m_s-1} }=k-\overline m}}{\PPP_\L^{\ssup\bc}\rk{\Nl=k}}
&\sim \frac{\frac 1\Nmin p(\frac{k- \overline m}{m_s})\ex^{-\lambda_1 \beta (k-\overline m)L^{-2}}}
{\frac 1\Nmin p(1)\ex^{-\lambda_1 \beta kL^{-2}}}\\
&\sim \ex^{\lambda_i \beta \overline m L^{-2}} \ex^{\gamma}p\Big(\Big(1-\sum_{i\in[s]}x_i\Big)/x_s\Big)\, .
\end{split}
\end{equation}
In the case $\lambda_1=0$, we have that $\qbc_{r}$ converges very quickly towards one, hence, using the asymptotics $\sum_{r=1}^N\frac 1r\sim\log N$, the first term on the right-hand side of~\eqref{011220231W} is
\begin{equation}
\ex^{-\sum_{r=m_s}^N\frac{1}{r}\qbc_r}\sim\ex^{-\sum_{r=m_s}^N\frac{1}{r}}\sim \frac {m_s}N.
\end{equation}
Furthermore, the second is asymptotic to $(|\L|\rhoe)^{-s}/\prod_{i\in[s]} x_i$, and the last one is
\begin{equation}
\frac{\PPP_\L^{\ssup\bc}\rk{{\rm N}_\L^{\sqsup{\Nmin+1,m_s-1} }=k-\overline m}}{\PPP_\L^{\ssup\bc}\rk{\Nl=k}}
\sim \frac{\frac 1{m_s} p(\frac{k- \overline m}{m_s})}
{\frac 1N p(1)}
\sim \frac N{m_s}\ex^{\gamma}p\Big(\Big(1-\sum_{i\in[s]}x_i\Big)/x_s\Big)\, .
\end{equation}
This finishes the proof of Proposition~\ref{prop:convergenceOfDensitiesPoissonDirichlet}.
\end{proof}

\appendix

\section{Proof of the remaining statements}\label{sec:remaining}
\begin{proof}[\textbf{\emph{Proof of Lemma \ref{lem:FKrepgamma}}}] We drop the super-indices for the boundary conditions during the proof, and keep $N,\beta,x,y$ fixed. By $\mu_{x,y}^{\ssup{\beta}}$ we denote the canonical Brownian bridge measure from $x$ to $y$ with time interval $[0,\beta]$ with boundary condition ${\bc}$ in the box $\L$. Its total mass is equal to $\mu_{x,y}^{\ssup{\beta}}(\1)=g_\beta(x,y)$ for $x,y\in\L$, where we write $\mu(f)$ for the integral of a function $f$ with respect to a measure $\mu$.

The proof of \eqref{eq:FK_Z} is well-known (see e.g.~\cite[Proposition 1]{ACK11}), but let us give some hints. The starting point is the Feynman--Kac formula 
\begin{equation*}
\begin{split}
    Z_N&=\frac 1{N!}\sum_{\sigma\in\mathfrak S_N}\int_\L\d x_1\dots\int_\L\d x_N\,\Big[\bigotimes_{i=1}^N \mu_{x_i,x_{\sigma(i)}}^{\ssup\beta}\Big](\1)\\
    &=\frac 1{N!}\sum_{\sigma\in\mathfrak S_N}\int_{\L^N}\d (x_1,\dots,x_N)\prod_{i=1}^N g(x_i,x_{\sigma_i})\, ,
\end{split}
\end{equation*}
which follows from an application of the symmetrisation operator $\Pi_+$ to the well-known trace formula for $\ex^{-\beta \Delta}$ (recall that $\Sfrak_{N}$ is the set of permutations $\sigma$ of $1,\dots,N$). Now decompose every permutation $\sigma$ into its cycles and use that $(g_\beta)_{\beta\in(0,\infty)}$ is a convolution semigroup, i.e., $\int_{\L} g_\beta(x,y) g_{\beta'}(y,z)\,\d y=g_{\beta+\beta'}(x,z)$ for all $x,z\in\L$ and $\beta,\beta'\in(0,\infty)$ (this is why \eqref{eq:FK_Z} does not hold for free boundary conditions). Iterating this $k-1$ times in a cycle of length $k$ gives that $\q_k$ is the contribution to this cycle. If $\partition_k(\sigma)$ denotes the number of cycles of length $k$ in $\sigma$, then the number of $\sigma$'s such that $\partition_k(\sigma)=\partition_k$ for all $k$ is given as $N!/\prod_{k\in\N}{k^{\partition_k}\partition_k!}$.

We now turn to the proof of \eqref{eq:firstformula}. Introducing the product measure
\begin{equation}
M_{u,v}^{\ssup {\beta,N}}=\bigotimes_{i=1}^N\mu_{u_i,v_i}^{\ssup\beta},\qquad u=(u_1,\dots,u_N), v=(v_1,\dots,v_N)\in \L^N,
\end{equation}
we can identify the kernel $\gamma_N$ (see \cite{Gin70}, and also \cite[Theorem 6.3.14]{a1981operator}) as
\begin{equation}\label{FKform}
    \gamma_N(x,y)=\frac{N}{Z_N N!}\sum_{\sigma\in\Sfrak_{N}}\int_{\L^N}M_{(x,u),\sigma(y,u)}^{\ssup{\beta,N}}[\1]\,\d u,
\end{equation}
where we write $\sigma(v_1,\dots,v_{N})=(v_{\sigma(1)},\dots,v_{\sigma(N)})$. 

In this expression, $N$ Brownian bridges connect the $N+1$ points $x,y,u_1,\dots,u_N$ in such a way that $x$ is only a starting site, $y$ is only a terminating site, and each $u_i$ is both. In particular, there is a bridge (concatenation of bridges with time horizon $[0,\beta]$) starting in $x$ and terminating in $y$. Say $r\in\{1,\dots,N\}$ is the length of such a path, containing $r-1$ points of $\{u_1,\dots,u_{N-1}\}$; we can then split the permutation of $\sigma\in\Sfrak_N$ into $\pi\in\Sfrak_r$ and $\sigma'\in\Sfrak_{N-r}$, and with the help of the Markov property of the bridges, rewrite
\begin{equation}
     \gamma_N(x,y)=\frac{1}{Z_N}\sum_{r=1}^N\frac{1}{(N-r)!} \, g_{r\beta}(x,y) \sum_{\sigma\in\Sfrak_{N-r}}\int_{\L^{N-r}} M_{u,\sigma(u)}^{\ssup{\beta,N-r}}[\1]\,\bigotimes_{i=1}^{N-r}\d u_i\, ,
\end{equation}
where we recall that $g_{k\beta}(x,y)$ is the total mass of $\mu_{x,y}^{\ssup{k\beta}}$. 
Abbreviate $A(\partition)=\prod_{k\in\N}\frac{1}{k^{\partition_k}\partition_k!}$. Similarly to \eqref{eq:FK_Z}, we then get
\begin{equation}\label{gammaFKexpansion}
\begin{aligned}
   \gamma_N(x,y)&= \frac{\sum_{r=1}^{N}g_{r\beta}(x,y)\sum_{\partition\in\Pfrak_{N-r}}{A(\partition)}\prod_{k=1}^{N-r}  \q_k^{\partition_k}}{\sum_{\partition\in\Pfrak_N}{A(\partition)}\prod_{k=1}^N \q_k^{\partition_k}}
   &=\sum_{r=1}^{N}g_{r\beta}(x,y)\frac{Z_{N-r}}{Z_N}\, ,
   \end{aligned}
\end{equation}
which ends the proof.
\end{proof}

\begin{proof}[\textbf{\emph{Proof of Proposition \ref{PropFreeEnergy}}}]
According to our representation in Lemma~\ref{lem-PoissonRepr}, 
\begin{equation}
    Z_N^{\ssup{\L_N,{\bc}}}(\beta)=\ex^{\al p_{\L_N}^{\ssup{\rm bc},N}}\Pbc_{\L_N,0}\rk{\Partn_{\L_N}=\rho\al_N}\, .
\end{equation}
For $\rho>\rhoc$, the probability term vanishes as $N\to\infty$, but exponentially fast in $|\L_N|$, as we proved in Lemma~\ref{PropsotionUpperBoundTotal} for  for the diffusive or periodic boundary condition. Hence,
\begin{equation}
{\rm f}(\rho)=-\frac 1\beta \lim_{N\to\infty}p_{\L_N}^{\ssup{\rm bc},N}=-\frac 1\beta p_0 ^{\ssup{\rm bc}}=-\beta^{-1-d/2}\zeta(1+\smfrac d2)\, ,
\end{equation}
as we can easily deduce from Lemma~\ref{CorollaryQkSupercritical}(ii) and (iii). The right-hand side is equal to the right-hand side of \eqref{freeenergyident}. For  free boundary condition, the entire argument is  even more immediate.

For $\rho<\rhoc$, we make a change of measure using \eqref{Changeofdensity}:
\begin{equation}
    \ex^{\al\presbc_{\L} }\Pbc_{\L,0}\rk{\Partn_\L=\rho\al}=\ex^{\al\presbc_{\L,\mu(\rho)}}\ex^{-\beta\mu(\rho)\rho\al}\Pbc_{\L,\mu(\rho)}\rk{\Partn_\L=\rho\al}\, .
\end{equation}
In Lemma~\ref{LemmaCLT}, we showed that $\Pbc_{\L,\beta,\mu(\rho)}\rk{\Partn_\L=\rho\al}$ is of order $\al^{-1/2}$ (i.e., vanishes, but not exponentially fast in $|\L|$) and hence
\begin{equation}
       {\rm f}(\rho)=-\frac 1\beta\lim_{N\to\infty}\big(\presbc_{\L,\mu(\rho)}-\beta \mu(\rho)\rho\big)= -\frac{1}{\beta}p_{\mu(\rho)} +\mu(\rho)\rho\, ,
\end{equation}
which is equal to the right-hand side of \eqref{freeenergyident}.
\end{proof}

\section*{Acknowledgements}
The authors would like to thank the anonymous referee for their constructive comments, which improved the
quality of this paper. AZ thanks Silke Rolles for making his research visit to TU München possible. QV would like to thank WK and the Weierstrass Institute for making his research visit possible.
WK is also affiliated to TU Berlin; QV is also affiliated to Alpen-Adria-Universität Klagenfurt; AZ is also affiliated to Universit\"at Potsdam.

This work was partially funded by Deutsche Forschungsgemeinschaft (DFG) through DFG Project no. 422743078 {\it The statistical mechanics of the interlacement point process} in the context of SPP 2265 {\it Random Geometric Systems}.

\bibliographystyle{alpha}
\bibliography{references}{}

\end{document}